\documentclass[]{interact}

\usepackage{epstopdf}
\usepackage[caption=false]{subfig}

\usepackage[numbers,sort&compress]{natbib}
\bibpunct[, ]{[}{]}{,}{n}{,}{,}
\renewcommand\bibfont{\fontsize{10}{12}\selectfont}

\theoremstyle{plain}
\newtheorem{theorem}{Theorem}[section]
\newtheorem{lemma}[theorem]{Lemma}

\newtheorem{proposition}[theorem]{Proposition}
\usepackage{xcolor}

\theoremstyle{definition}

\usepackage{multirow}

\theoremstyle{remark}

\newcommand{\ecc}{\`{e}\ }
\newcommand{\acc}{\`{a}\ }

\newcommand{\ucc}{\`{u}\ }

\newtheorem{defin}{Definition}
\newtheorem{thm}{Theorem}
\usepackage{enumitem}
\usepackage{graphicx}
      \usepackage{amsmath,amssymb,euscript,amsfonts,mathrsfs,amscd}
      
\begin{document}

\articletype{Original Research Article}

\title{Loss Functions in Restricted Parameter Spaces and Their Bayesian 
Applications~\footnote{This is a final submitted version of the paper published in Journal of Applied Statistics http://dx.doi.org/10.1080/02664763.2019.1586848}}

\author{
\name{P. Mozgunov\textsuperscript{a}\thanks{Contact P. Mozgunov, Email:p.mozgunov@lancaster.ac.uk}, T. Jaki\textsuperscript{a} and M. Gasparini\textsuperscript{b}}
\affil{\textsuperscript{a}Department of Mathematics and Statistics, Lancaster University, UK.; \textsuperscript{b}Dipartimento di Scienze Matematiche, Politecnico di Torino, Turin, Italy}
}

\maketitle

\begin{abstract}
Squared error loss remains the most commonly used loss function for constructing a Bayes estimator of the parameter of interest. However, it can lead to sub-optimal solutions when a parameter is defined on a restricted space. It can also be an inappropriate choice in the context when an extreme overestimation and/or underestimation results in severe consequences and a more conservative estimator is preferred. We advocate a class of loss functions for parameters defined on restricted spaces which infinitely penalize boundary decisions like the squared error loss does on the real line. We also recall several properties of loss functions such as symmetry, convexity and invariance. We propose generalizations of the squared error loss function for parameters defined on the positive real line and on an interval. We provide explicit solutions for corresponding Bayes estimators and discuss multivariate extensions. {Four} well-known Bayesian estimation problems are used to demonstrate inferential benefits the novel Bayes estimators can provide in the context of restricted estimation. 
\end{abstract}


\section{Introduction}

In many parameter estimation problems, the support of the parameter is either naturally restricted (e.g. probability, variance, exponential distribution parameter) or an investigator can restrict it based on previously obtained knowledge (e.g. treatment effects on children given the data for adults). The knowledge about restricted space can carry important information and improve estimation~\citep{admissible}. There are also application areas in which estimates on bounds of the restricted space are highly undesirable as they can lead to severe consequences. For instance, underestimating the potential of an event to have disastrous or life-threatening consequences may be worse than overestimating it~\citep{norstrom1996}. An erroneously low estimated risk-level can lead to the absence of initiative to reduce it.

One of the ways to incorporate the information about a restricted space is to employ a Bayesian approach and to define a (uniform) prior distribution on the restricted space~\citep{hartigan}. However, once a posterior is obtained, the squared error loss
\begin{equation}
L_q(\theta,d)= (\theta-d)^2
\label{quad}
\end{equation}
where $d$ is a decision the statistician has to take in order to
approximate an unknown estimand $\theta$, called parameter, is often used to summarize a posterior distribution. The squared loss function~(\ref{quad}) ignores the information about the restricted parameter space and is recognised to lead to suboptimal solutions~\citep[see e.g.][for alternative loss functions for a scale parameter]{stein,brown}. Research on improving the Bayes estimator under squared loss function (a posterior mean) for a scale  parameter has consequently attracted a great deal of attention~\citep[see e.g.][and references therein]{kubokawa}.

In addition, since loss function~(\ref{quad}) does not penalise boundary values, it was found to be unacceptable in many application areas: see~\cite{norstrom1996} for examples in  reliability analysis, \cite{karimnezhad2014bayes} in environmental sciences and \cite{sainthilary2018} in  drug development. To avoid boundary values of a scale parameter, \cite{norstrom1996} introduced the \textit{precautionary loss function}
\begin{equation}
L_{sq}(\theta,d) = \frac{(d-\theta)^2}{d} \ {\rm where} \ \theta,d \in (0,+\infty)
\label{lossscale}
\end{equation}
which was used by many researchers~\citep{prediction,karimnezhad2016bayes}.

The precautionary loss function covers the case of the scale parameter. There are, however, many applications in which the parameter of interest is restricted to an interval $(a,b)$ and similar problems of severe consequences of boundary decisions can appear. {We provide two motivating examples from the medical domain. Firstly, in the setting of outbreaks, the probability of response for a drug able to stop the outbreak should be high (say $>90\%$). In this case, overestimation of the probability of response can lead to the approval of a drug which cannot stop the outbreak that can cost a lot of human lives. Secondly, in many paediatric trials, adult data  responses can be used to define feasible values of responses (usually an interval) for children. At the same time, underestimation of the response effect for comparative treatments in paediatric clinical trials is highly undesirable as it might result in an underpowered and unethical study. In both settings, one can benefit from the application of specific loss function for parameters defined on an interval. We provide more details on the consequences in the later example in Section 6.}

Despite the importance, the question of an appropriate loss function choice for a parameter $\theta$ defined on the interval $(a,b)$ has been paid less attention in the statistical literature compared to a scale parameter. At the same time, its importance is acknowledged in many fields~\citep[see][for examples in compositional data analysis]{ait,logisticnorm2}. {Specifically, \cite{ait} proposed to use
$$
L_{iB}(\theta,d)=\left( {\rm logit}(d) - {\rm logit}(\theta) \right)^2.
$$
as the measure of distance for $d,\theta \in (0,1)$ where ${\rm logit}(x)=\log\frac{x}{1-x}$ is the logit-transformation. This is the squared error loss after the logit-transformation of $d$ and $\theta$. While being intuitively clear, it is not convex and it has no explicit formula of the Bayes estimator, making its use challenging in applications.} Furthermore, despite the variety of literature on families of loss functions for parameters restricted to an interval and on corresponding improved Bayes estimators~\citep[see e.g.][]{kubokawa,restricted}, they seem to be rarely applied in practice due to  their complexity and to a lack of closed-form solutions. The choice of loss functions for parameters defined on the interval and on the positive real line is yet an under-represented area in the Bayesian literature and the usual mean still remains a common summary statistic.

The contribution of this work is twofold.  Firstly, we provide a unified approach to define symmetry of a loss function when a parameter space is restricted to a particular open subset based on an appropriate definition of distance. We underline that our distances on corresponding parameter spaces share a common property - infinite penalization of the bounds which is also known as the \textit{balance property}~\citep{admissible}. We also recall two other desirable properties of loss functions: convexity and invariance. Secondly, we propose several loss functions which are as simple as the squared loss function~(\ref{quad}), have explicit solutions for the corresponding Bayes estimator and incorporate the information about the restricted parameter space in the corresponding Bayes estimator. In particular, we propose the scale invariant generalization of the the precautionary loss function for a scale parameter $\theta \in (0,+\infty)$ and the interval squared loss function
$$
L_{iq}(\theta,d)=\frac{(d-\theta)^2}{(d-a)(b-d)}
$$
for the parameter $\theta \in (a,b)$. We show that the Bayes estimator corresponding to the interval squared loss function includes the Bayes estimator of the squared loss function~(\ref{quad}) and of the precautionary loss function as limiting cases. It is found that the interval squared and precautionary loss functions are both symmetric on the corresponding parameter spaces and can be useful in application areas where conservative estimates are preferred. We generalise the approach for the multivariate parameter space and demonstrate how Bayes estimators obtained using the proposed loss functions behave in three classic problems of Bayesian estimation compared to standard approaches. 

The rest of the paper is organized as follows. A historical perspective for the scale parameter estimation and the case of symmetric loss function on the positive real line is given in Section~2. Section~3 introduces the novel loss function for an interval. The multivariate generalizations are given in Section~4. Three examples  demonstrating  novel loss functions and corresponding Bayes estimators are considered in Section~5.

\section{Scale Symmetry}

\subsection{A Historical Anecdote: Galileo on Scale Symmetry}

In the Spring of 1627, a {\em peculiar controversy}
\footnote{the italic is a translation of a commentary to \cite{galilei}
appearing in the edition of Galilei's works 
mentioned in the bibliography,
from which all of the quotes are taken, following \cite{scardovi}.}
arose in one of Florence intellectual circles,
where {\em noble gentlemen} used to entertain
{\em erudite talks}:  \\

\begin{em}
 Un cavallo, che vale veramente cento scudi,
da uno \ecc stimato mille scudi e da un altro dieci scudi:
si domanda chi abbia di loro stimato meglio,
e chi abbia fatto manco stravaganza nello stimare.  \\
\end{em}  

The problem translates into: ``A horse, whose true worth is one hundred {\em scudi}~\footnote{a monetary unit, literally, a shield}, is
estimated by someone to be one thousand {\em scudi} and by someone
else to be ten {\em scudi}: the question is, who gave a better
estimate, and who instead gave a more extravagant estimate?''. It is
formulated in a letter from Andrea Gerini to Nozzolini,
an {\em erudite priest}. Gerini wanted Nozzolini's opinion on 
a sentence by \cite{galilei}, according to whom  \\

\begin{em}
 \dots li due stimatori abbiano egualmente esorbitato
e commesse eguali stra\-va\-gan\-ze nello stimare l'uno mille
e l'altro dieci quello che realmente val cento,  \\ 
\end{em}  

\noindent which translates to: ``The two estimators have been equally exorbitant and are responsible for an equal extravagance by estimating, one thousand the former and ten the latter, what is really worth one hundred''. 

In the  intense correspondence following the initial letters, 
Nozzolini argues that the estimates should be evaluated 
according to the {\em arithmetic
proportion}, whereas Galileo insists 
that the correct method of judging is by 
{\em geometric proportion}.
The crux of the problem is that the estimand is a 
positive quantity, for which the {\em geometric proportion}
seems more  appropriate,
as wittingly argued by Galileo in another
letter:   \\

\begin{em}
 Se uno stimasse alta dugento braccia
una torre, che veramente fusse alta cento,
con quale esorbitanza nel meno paregger\acc il signor Nozzolini
\l'altra nel pi\ucc?   \\
\end{em}

\noindent which translates as: ``If one were to overestimate a one-hundred arm high tower as two-hundred arm high, what underestimate would 
Nozzolini consider as equally deviating?''

\subsection{Scale Symmetry and Scale Invariance}
Consider a toy example to illustrate Galileo's position in modern statistical terms. Two inferential procedures are based on two
independent experiments:
 \begin{enumerate}
 \item Estimate $\mu \in (-\infty,+\infty)$ given i.i.d. $X_i \sim \mathcal{N}(\mu,\sigma^2)$, $\sigma^2$ is known;
 \item Estimate  $\sigma \in (0, + \infty)$  given i.i.d. $Y_i \sim \mathcal{N}(\mu,\sigma^2)$, $\mu$ is known.
 \end{enumerate}

\noindent Assume that the true parameters values are equal $\theta=\mu=\sigma$ and $X_i$ and $Y_k$
 independent for all $i,k$. Using squared error loss~(\ref{quad}),
 the decision $\mu=0$ in the first experiment and $\sigma=0$ in the
 second are equally penalized, while this should not be the case.  The claim of $\sigma=0$ implies that the $Y$'s are degenerate
 random variables, an extremely strong statement which should be
 penalized similarly to the decision $\mu=+\infty$ or
 $\sigma=+\infty$. The squared error loss function imposes an infinite
 penalty to a boundary decision in the first experiment and does not
 in the second one. While the decision $\sigma=0$ is usually prevented by a proper choice of the prior, the squared loss function does not imply that it should be avoided and associated with a severe penalty. In contrast, an appropriate loss function imposes such a penalty and can be also used to prevent boundary decisions. We define the properties of such loss function for a scale parameter in this section.

Let us start with the following definition for a parameter defined on the whole real line.
\begin{defin}
  A loss function $L(\theta,d)$ is symmetric if, for every
$d_1$, $d_2$ and $\theta$
  $\in \mathbb{R}^1$
\begin{equation}
\left( \theta-d_1\right)^2 = \left( d_2-\theta\right)^2
\label{symmetry}
\end{equation}
implies $L(\theta,d_1)=L(\theta,d_2)$.
\label{defsymmetry}
\end{defin}
The Definition~\ref{defsymmetry} implies that two decisions defined on the real line should be equally penalized by a symmetric loss function $L(\cdot,\theta)$ if they stand on the same squared distance from $\theta$. Note that for $d_1<d_2$ Equation~(\ref{symmetry}) can be rewritten as
$ \theta =\left(d_1+d_2\right) /2.$ It follows that if $\theta$ is the {\em arithmetic mean} of $d_1$ and $d_2$,
then these decisions should be equally penalized. 
Clearly, the squared error loss of equation~(\ref{quad}) is symmetric on the real line by definition. 

Then, Galilei's claim of {\em eguali stravaganze} for a positive
parameter $\theta$ can be expressed in
modern terminology as the requirement of a {\em scale symmetric}
loss function, as in the following definition.
\begin{defin}
  A loss function $L(\theta,d)$ is scale symmetric if, for
  every 
$d_1$, $d_2$ and
  $\theta$ $\in \mathbb{R}_{+}$
\begin{equation}
\frac{d_1}{\theta}=\frac{\theta}{d_2}
\label{scalesymmetry}
\end{equation}
implies $L(\theta,d_1)=L(\theta,d_2)$.
\label{defscale}
\end{defin}
\noindent Equation~(\ref{scalesymmetry}) can be rewritten as
$ \theta =\sqrt{d_1 d_2}$
or 
$
\log(\theta) = \left(\log d_1 + \log d_2\right) / 2.
$
In other words, if $\theta$ is the {\em geometric mean} of $d_1$ and $d_2$,
then these decisions should be equally penalized by a scale symmetric loss function. As in Definition~\ref{defsymmetry} of symmetry for the real line, two decisions are symmetric if the parameter $\theta$ is their appropriate mean -- geometric in this case as opposed to arithmetic. This fact will be used for our proposal of the definition of symmetry on interval in Section~3.

The distance on the positive real line, $\mathbb{R}_{+}$ defined as~\citep{logisticnorm2}
\begin{equation}
\mathcal{D}_{+}(\theta,d)=(\log \theta - \log d)^2
\label{distancealt}
\end{equation}
is known in Statistics as Brown's loss function~\citep{brown}. Its motivation is to rescale the positive real line to the whole real line via the log transformation and to use the squared error loss function. Here, the logarithm function is a natural choice for a positive random variable.  Note that $
\mathcal{D}_{+}(\theta,d_1)=\mathcal{D}_{+}(\theta,d_2)
$
implies either $d_1=d_2$ or Equation~(\ref{scalesymmetry}). Therefore, we could also restate Definition~\ref{defscale} in terms of~$\mathcal{D}_{+}(\cdot)$.

The Euclidean distance on the real line and $\mathcal{D}_{+}$ on the positive real line infinitely penalize boundary values on the corresponding parameter space. In case of $\theta \in \mathbb{R}$, the squared distance $L_q(\theta,d)$ takes an infinite value when $d=\pm \infty$.  For similar reasons, we require
that an appropriate loss function  for a scale parameter should go to infinity as the decision approaches
the natural boundaries of the parameter space, to reproduce the
behaviour at $\pm \infty$ of the squared error loss function. A loss function with this property are also called \textit{balanced}~\citep{admissible}. 

We recall one more property of loss functions for a parameter on the positive real line - the scale invariance.
\begin{defin} Loss function $L(\theta,d)$ is scale invariant if for every $c>0$ and every pair $(\theta,d)$, $$L(\theta,d) = L(c\theta,cd). $$
\end{defin}
Then, the following result can be obtained.
\begin{lemma}
A loss function is scale invariant and scale symmetric if and only if
it can be written as a scalar function $g$ such that 
$g(d/\theta)=g(\theta/d)$.
\label{lemma}
\end{lemma}
\begin{proof}
A loss function $L(\theta,d)$ is scale invariant if and only if it is
ratio-based, i.e. if and only if there exists a scalar function $g(x),
x>0$ such that $L(\theta,d)=g(d/\theta)$. 
Scale invariance therefore implies
$L(\theta,d)=g(d/\theta)$
for some $g$, whereas by Definition~(\ref{defscale}) 
scale symmetry implies
$L(\theta,d)=L(\theta,\theta^2/d)$ and vice versa. 
\end{proof}

It follows that the squared loss function~(\ref{quad}) is not scale symmetric, is not scale invariant and does not penalize all boundaries for a scale parameter.

\subsection{Symmetric Loss Functions on the Positive Real Line}

In the modern statistical literature, the inadequacy of
difference-based loss functions, like the squared error loss, for
estimating certain positive quantities has often been recognized~\citep{stein1961,stein}. Several alternative loss functions have been proposed, the
best-known being the normalized squared loss function proposed by~\cite{stein}
$$
L_{nq} (\theta,d) = \left(\frac{d}{\theta} -1\right)^2, 
$$
Stein's loss (or an entropy loss function) 
$$
L_S (\theta,d) = \frac{d}{\theta} - 1 - \log \left( \frac{d}{\theta} \right)
$$
and Brown's loss function~\citep{brown} itself, $\mathcal{D}_{+}(\theta,d)$. One can check that all of functions above are scale invariant, but only Brown's loss function is scale symmetric and infinitely penalizes the boundary decisions.  Unfortunately, Brown's loss function is not {\em convex}, a feature of loss functions which is often required to represent risk aversion and for the sake of regularizing the associated minimization problems. Another unpleasant consequence of non-convexity is that the Bayes estimator associated to Brown's loss function is usually difficult to calculate. Below we propose simple alternative loss functions which share the desirable properties of a loss function on the positive line and have explicit Bayes estimators.

We propose a family of loss functions defined for $k>0$ as
\begin{equation}
L_k (\theta,d) = \left( \frac{d}{\theta} \right)^k + 
\left( \frac{\theta}{d} \right)^k  -2
\label{mauroloss}
\end{equation}
which are scale symmetric, scale invariant, convex, and which tend to infinity at the boundaries. {Expression~(\ref{mauroloss}) is a function of the ratio $\frac{d}{\theta}$ to make it scale symmetric, and it satisfies Lemma~\ref{lemma} to make it scale invariant.} {The constant 2  is subtracted so the minimum value of the loss function $L_k=0$ is attained at $d=\theta$.} In this paper, we focus on the case $k=1$
\begin{equation}
L_1(\theta,d) = \frac{(d-\theta)^2}{\theta d}
\label{scaleinvariantloss}
\end{equation}
which can be considered as a modification of the squared error loss function. The {numerator} is again the squared distance, but the denominator guarantees the infinite penalization for $d=0$. It is easy to see that the loss function~(\ref{scaleinvariantloss}) is a scale invariant version of the precautionary loss function~(\ref{lossscale}). 



\subsection{Scale Means (the Minimizers) and Scale Variances}
 
Within the Bayesian approach, $\theta$ is a random {variable}
with a distribution which conveys the uncertainty the researcher 
has in a given state of information (whether prior,
posterior, elicited, objective and so on).
In such a scenario, a point summary of the distribution of 
$\theta$ minimizing the risk (i.e. the expected loss)
associated with a given loss function is often required.
Such a minimizer of an expected $d$ is usually called a Bayes estimator.
When a scale symmetric loss function is used, we propose to call
such minimizers \textit{scale means.} 
In case of convex loss functions, such as the novel ones 
listed in the previous section, 
minimization can be performed explicitly, as in Theorem~\ref{scalethm}.

\begin{thm}
\label{scalethm}
Let $\theta$ be a positive random variable with a posterior density function $f$ and such that $\mathbb{E}(\theta^k)<\infty$ and $\mathbb{E}(\theta^{-k})<\infty$, where $\mathbb{E}$ denotes the posterior mean with respect to $f$ and $k>0$. Then,

\textbf{(a)}  Expectation of the loss function $L_k (\theta,d)$~(\ref{mauroloss}) with respect to~$f$ is minimized by the Bayes estimator (scale mean)
\begin{equation}
\label{estimate}
\hat{d}_k=\left(\frac{\mathbb{E}(\theta^k)}{\mathbb{E}(\theta^{-k})}\right)^{\frac1{2k}}.
\end{equation}

\textbf{(b)} Expectation of the precautionary loss function $L_{sq}$~(\ref{lossscale}) is minimized by the Bayes estimator (scale mean)
\begin{equation}
\hat{d}_{sq}=\sqrt{\mathbb{E}(\theta^2)},
\label{scaleestimator}
\end{equation}
for which the following bound holds:
$\hat{d}_{sq} \geq \mathbb{E}(\theta)$.
\label{theoremscale}
\end{thm}

\begin{proof}
{\textbf{(a)} The expectation of the loss function~(\ref{mauroloss}) with respect to the posterior density function $f$ takes the form
$$\mathbb{E}(L_k(\theta,d))= \mathbb{E} \left(\frac{d}{\theta} \right)^k + \mathbb{E} \left( \frac{\theta}{d} \right)^k -2 = d^k \mathbb{E} \left( \theta^{-k} \right) + d^{k-1} \mathbb{E} \left( \theta^k \right)-2.$$
Then, the decision $d$ minimising the expected loss function is found solving
$$ \frac{ \partial \mathbb{E}(L_k(\theta,d))}{\partial d} = k d^{k-1} \mathbb{E}\left(\theta^{-k} \right) - k d^{-k-1} \mathbb{E} \left(\theta^k \right) =0$$
This results in
$\hat{d}_k=\left(\frac{\mathbb{E}(\theta^k)}{\mathbb{E}(\theta^{-k})}\right)^{\frac1{2k}},$
and in the special case of $k=1$,
$\hat{d}_1~=~\sqrt{\mathbb{E}(\theta)/\mathbb{E}(\theta^{-1})}.$}

{\noindent \textbf{(b)} Similarly to the previous point, the expectation of the precautionary loss function~(\ref{lossscale}) with respect to the posterior density function $f$ taken the form
$$\mathbb{E}(L_{sq}(\theta,d))= \mathbb{E} \left( \frac{(d-\theta)^2}{d} \right)= \frac{d^2-2d\mathbb{E}(\theta) + \mathbb{E}(\theta^2)}{d}.$$
Then, the decision $d$ minimising the expected loss function is found by
$ \frac{ \partial \mathbb{E}(L_{sq}(\theta,d))}{\partial d} = 0.$
This results in
$ \hat{d}_{sq}=\sqrt{\mathbb{E}(\theta^2)}.$  Using Jensen inequality for $\theta^2$ one can obtain 
$\mathbb{E}\left(\theta^2 \right) \geq \mathbb{E}^2\left( \theta\right).$
Applying the squared root to both sides of the inequality the result immediately follows.}
\end{proof}



In a more fundamental Bayesian approach, a Bayes estimator is regarded
only as a convenient summary of the posterior, and a loss function as
a way to prescribe what kind of summary is appropriate. Typically, a
posterior expectation is used as the Bayes estimator, implying that a
squared loss function is being used. A second step is usually taken
to accompany the Bayes estimator with a measure of uncertainty of the
posterior.  If a posterior mean is used, a posterior variance is
usually presented. However, if a scale symmetric loss function is considered
to be a reasonable criterion for choosing an estimator, i.e. a number
which minimizes a posterior expected loss, then it is also reasonable
to present the achieved minimum of the posterior expected loss as a
second summary of the posterior. For the given loss functions~(\ref{mauroloss}) and~(\ref{lossscale}), particularly simple expected
posterior losses can be obtained. 
In particular, for the loss function~(\ref{mauroloss}) such
{\em scale variance}
of order $k$ of the random
variable $\theta$ in Theorem~\ref{theoremscale}(a) can be written 
$
\hat{\tau}_k (\theta) := 2\sqrt{\mathbb{E}(\theta^k) \mathbb{E}(\theta^{-k})} - 2,
$
whereas the scale variance 
for the precautionary loss function~(\ref{lossscale}) in Theorem~\ref{theoremscale}(b) 
is
$\hat{\tau}(\theta)=2 \left(\sqrt{\mathbb{E}(\theta^2)} - \mathbb{E}(\theta) \right). $

\section{Interval Symmetry}
The approach used above for a positive parameter can be generalized to the
parameter defined on the interval
$(a,b)$. The issue of a restricted parameter space is not
usually discussed in the choice of the loss function and corresponding Bayes estimator: bounds are taken into account through the prior specification only, then the squared loss function and posterior mean (the corresponding Bayes estimator) are used~\citep{restricted}. Such solutions can be suboptimal if boundary decisions are to be avoided. Below, we define the property of the symmetry on an interval and show that the novel definition generalizes the cases of parameters on the whole real line and on the positive real line. We provide the loss function with desirable properties which is, again, a generalization of the squared loss function and the precautionary loss function.

\subsection{Symmetric Loss Functions on Intervals}
Let us consider an inferential problem for which the parameter of interest lies
in a particular interval $(a,b)$. 
Define the following transformation
\begin{equation}
{\rm logit}_{(a,b)}(x)=\log \frac{x-a}{b-x}
\label{logitab}
\end{equation}
where $a<x<b$. Notice that, for $a=0$ and $b=1$, transformation
(\ref{logitab}) reduces to the common logit transformation {widely used in Statistics, and was used by~\cite{ait} to justify the definition of distance on the unit interval.}
%
%
Following the same lines of reasoning as in Section 2.2, 
we {use this transformation} to introduce the definition
of  {\em symmetric on the interval $(a,b)$} loss function (or
simply {\em interval symmetric} loss function), {and demonstrate why it is a convenient choice.}
\begin{defin}
  A loss function $L(\theta,d)$ is symmetric on the interval
  $(a,b)$ if, for every choice of 
$d_1, d_2 \in (a,b)$
  and $\theta$ $\in (a,b)$ 
\begin{equation}
{\rm logit}_{(a,b)}(\theta)=\frac{{\rm logit}_{(a,b)}(d_1) + {\rm logit}_{(a,b)}(d_2)}{2}.
\label{intervalsymmetry}
\end{equation}
implies $L(\theta,d_1)=L(\theta,d_2)$.
\label{definterval}
\end{defin}
In other words, two decisions $d_1$ and $d_2$ should be penalized
equally if the mean of their logit transformation is equal to the
logit transformation of $\theta$. Lemma~\ref{equiva}  justifies the use of the logit transformation~(\ref{logitab}).
\begin{lemma}
Definition \ref{definterval} is equivalent to Definition
\ref{defsymmetry} when $a~\to~-\infty$ and 
$b~\to~+\infty$ and equivalent  to 
Definition \ref{defscale} when $a~\to~0$ and $b~\to~+\infty$.
\label{equiva}
\end{lemma}
\begin{proof}
Condition (\ref{intervalsymmetry}) for $a<d_1<d_2<b$ can be rewritten
$$\theta=f(a,b,d_1,d_2) \equiv \frac{ab -d_1d_2 + \sqrt{(d_1-a)(b-d_1)(d_2-a)(b-d_2)}}{a+b-d_1-d_2}.$$
Obviously, $\theta$ is a symmetric function of $d_1$ and $d_2$. Considering two limits
$$\lim_{a \to -\infty, \ b \to +\infty}f(a,b,d_1,d_2)= \frac{d_1+d_2}{2}, \ \lim_{a \to 0, \ b \to +\infty}f(a,b,d_1,d_2)= \sqrt{d_1 d_2},$$
it can be easily seen that the definitions are equivalent. 
\end{proof}
\noindent It follows from Lemma~\ref{equiva} that Definition~\ref{definterval}
is a convenient generalization of the definition of 
symmetry and of scale symmetry.


\subsection{An Interval Symmetric Loss Function
and a Bayes estimator}

As in the case of a positive parameter and scale symmetric loss
functions, discussed in Section 2, the approach of
\cite{brown} of specifying a squared loss function 
after rescaling the interval $(a,b)$ to the real
line via, for example, the logit transformation~(\ref{logitab})
provides the loss function 
\begin{equation}
L_{iB}(\theta,d)=\left( {\rm logit}_{(a,b)}d - {\rm logit}_{(a,b)} \theta \right)^2.
\label{aitdistance}
\end{equation}
On the unit interval this loss function  is equivalent to so-called Aitchison distance proposed by~\cite{ait} for parameters defined on a simplex. However, the loss function~(\ref{aitdistance}) is not convex and its minimization problem does not have an explicit 
solution. As an alternative, we propose the following loss function 
\begin{equation}
L_{iq}(\theta,d)=\frac{(d-\theta)^2}{(d-a)(b-d)}.
\label{intervalloss}
\end{equation}
which is interval symmetric and tends to infinity when the decision $d$ 
tends to bounds  $a$ and $b$. 
     
Note that loss function~(\ref{intervalloss}) for $a=0$ and $b=1$ looks similar
to the well-known loss function
$(d-\theta)^2/(\theta(1-\theta))$
which, however, does not penalise boundary decisions. 

 
\noindent The Bayes estimator corresponding to $L_{iq}$ is given in Theorem~\ref{intervaltheorem}.

\begin{thm}
Let $\theta \in (a,b)$ be a random variable with a posterior density function $f$ and
$\mathbb{E}(\theta^2)<\infty$ where $\mathbb{E}(\cdot)$ denotes the expectation with respect to $f$. Then,

\textbf{{(a)}} the expectation of the interval symmetric loss function $L_{iq}$~(\ref{intervalloss}) with respect to~$f$ is  minimized by the Bayes estimator
\begin{equation}
  \hat{d}_{iq}=\frac{ab-\mathbb{E}(\theta^2)+\sqrt{(\mathbb{E}(\theta^2)-ab)^2
-(a+b-2\mathbb{E}(\theta))(2ab\mathbb{E}(\theta)-(a+b)
\mathbb{E}(\theta^2))}}{a+b-2\mathbb{E}(\theta)}
\label{intervalestimator}
\end{equation}

\textbf{{(b)}} In the limiting case $a \to -\infty$ and $b \to
+\infty$ estimator~(\ref{intervalestimator})~minimizes the expectation of squared
loss function~(\ref{quad}), and in the limiting case $a \to 0$ and $b \to +\infty$ estimator~(\ref{intervalestimator}) minimizes the expectation of precautionary loss function~(\ref{lossscale}).
\label{intervaltheorem}
\end{thm}
\begin{proof}

\textbf{{(a)}} The equality is proved by differentiating in $d$ the expected losses $L_{iq}$~(\ref{intervalloss}).

 \textbf{{(b)}} Denote the estimator~(\ref{intervalestimator}) 
by $d \equiv g(a,b,\theta)$; then, taking the limits
$$\lim_{a \to -\infty, \ b \to +\infty}g(a,b,\theta)= \mathbb{E}(\theta), \ \lim_{a \to 0, \ b \to +\infty}g(a,b,\theta)= \sqrt{\mathbb{E}(\theta^2)},$$
it is easy to see that the obtained estimators are equivalent to
the minimizers of the squared loss function~(\ref{quad}) and of the precautionary loss function~(\ref{intervalloss}), respectively. Note that $\hat{d}_{iq} \to \frac{a+b}{2} \ {\rm as} \ \mathbb{E}(\theta) \to \frac{a+b}{2}$.
\end{proof}

It follows from Theorem~\ref{intervaltheorem} that the Bayes estimator $\hat{d}_{iq}$ includes the Bayes estimator under squared loss function~(\ref{quad}) and precautionary loss function~(\ref{lossscale}) as special cases. 

\section{Multivariate Generalizations}

The definition of symmetry can be generalized to the 
case of a parameter belonging to a subset of $\mathbb{R}^{m}$
by applying the same ideas to selected shapes of the parameter
space as in the following definition.

\begin{defin}
\label{multidef}
Let $\boldsymbol{\theta}=(\theta^{(1)},\theta^{(2)},\ldots,\theta^{(m)})^{\rm T}$ be a parameter lying in one of the 
parameter spaces $\Theta \subset \mathbb{R}^m$ listed below. Let $\boldsymbol{d_i}=(d_i^{(1)},d_i^{(2)},\ldots,d_i^{(m)})^{\rm T}$, $i=1,2$ be two vectors of decisions defined on the same parameter space.
A loss function $L(\boldsymbol{\theta}, {\bf d})$ is a
 multivariate $\Theta-$symmetric if the equality
$$L(\boldsymbol{\theta},{\bf d_1}) = L(\boldsymbol{\theta},{\bf d_2})$$
is implied by each triple  $\boldsymbol{\theta}, {\bf d_1}, {\bf d_2} \in \Theta$ satisfying the following respective definitions of distances:
\renewcommand\labelenumi{\theenumi}
\renewcommand{\theenumi}{(\alph{enumi})}
\begin{enumerate}
\item when $\Theta = \mathbb{R}^m$  (symmetry on $\mathbb{R}^m$ itself):  
$$\sqrt{ \sum_{j=1}^m \left( d_1^{(j)}-\theta^{(j)} \right)^2}=\sqrt{ \sum_{j=1}^m \left( d_2^{(j)}-\theta^{(j)} \right)^2};$$

\item
when $\Theta = \mathbb{R}^m_{+} = \{\theta: \theta^{(i)} >0,
  i=1, \ldots, m\}$ (scale symmetry on $\mathbb{R}^m_{+}$):
$$
\sqrt{ \sum_{j=1}^m \log^2 \left( \frac{d_1^{(j)}}{\theta^{(j)}} \right)}=\sqrt{\sum_{j=1}^m \log^2 \left( \frac{d_2^{(j)}}{\theta_j} \right)};
$$

\item
 when $\Theta = \{\theta: (a_1 < \theta^{(1)} < b_1), \ldots, 
(a_m < \theta^{(m)}  < b_m) \}$ (symmetry on  an $\mathbb{R}^m$-rectangle):
$$\sqrt{  \sum_{j=1}^m \left( {\rm logit}_{(a_j,b_j)} d_1^{(j)}-{\rm logit}_{(a_j,b_j)}\theta^{(j)} \right)^2}=\sqrt{  \sum_{j=1}^m \left( {\rm logit}_{(a_j,b_j)} d_2^{(j)}-{\rm logit}_{(a_j,b_j)}\theta^{(j)} \right)^2};$$ 
\item
when $\Theta = \{\theta:  \theta^{(1)}>0, \ \theta^{(2)}>0, \ \ldots,
\theta^{(m)}>0; \sum_{i=1}^m \theta^{(i)}=1 \}$ 
(symmetry on the unit simplex):
$$\sqrt{ \frac{1}{m} \sum_{i<j} \left(\log \frac{d^{(i)}_1}{d^{(j)}_1} - \log \frac{\theta^{(i)}}{\theta^{(j)}} \right)^2 }=\sqrt{ \frac{1}{m} \sum_{i<j} \left(\log \frac{d^{(i)}_2}{d^{(j)}_2} - \log \frac{\theta^{(i)}}{\theta^{(j)}} \right)^2 }.$$ 

\end{enumerate}

\end{defin}
\noindent The definition of the symmetric loss function in each case employs a distance corresponding to the particular restricted space. While the distances in (a)~-~(c) are natural extensions of the previously used, the definition in (d) is less straightforward.  Definition~\ref{multidef}(d) uses the Aitchison distance proposed by~\cite{ait} and employed in compositional data analysis. Regarding properties of the proposed definition, Lemma~\ref{equivamulti}, similar to Lemma~\ref{equiva}, holds.

\begin{lemma}
\label{equivamulti}
  Let $\boldsymbol{\theta}=(\theta^{(1)},\theta^{(2)},\ldots,\theta^{(m)})^{\rm T}$ be a vector
  of parameter of interest such that $\theta^{(1)} \in (a_1,b_1), \
  \theta^{(2)} \in (a_2,b_2), \ldots, \theta^{(m)} \in (a_m,b_m)$ and $\boldsymbol{d_i}=(d_i^{(1)},d_i^{(2)},\ldots,d_i^{(m)})^{\rm T}$ be a vector of corresponding
  decisions lying in corresponding intervals.  
Definition~\ref{multidef}(c) is equivalent to Definition~\ref{multidef}(a)
  when $a_i \to - \infty$ and $b_i \to \infty$ for all $i=1, \ldots,m$
  and to Definition~\ref{multidef}(b) when $a_i= 0$ and $b_i \to
  \infty$ for all $i=1, \ldots,m$.
\end{lemma}

\noindent Following~\cite{brown}, all distances in Definition~(\ref{multidef}) could be taken as corresponding symmetric loss functions. For example, in case \textit{(b)}, one could define
\begin{equation}
\mathcal{D}_{+}^{(m)} \left( \boldsymbol{\theta}, \boldsymbol{d} \right) = \sum_{j=1}^{m}  \log^2 \left( \frac{d^{(j)}}{\theta^{(j)}} \right).
\end{equation}
\noindent At the same time, some convex alternatives could be considered when leading to simple solutions of minimization problems. However, the search of the symmetric multivariate generalization of our proposed loss functions $L_k$, $L_{sq}$ and $L_{iq}$ seems to be a non-trivial one. We propose the following loss functions for parameters with non-negative components.
\begin{proposition}
Let $\boldsymbol{\theta}=(\theta^{(1)},\ldots,\theta^{(m)})^{\rm T} \in
  \mathbb{R}^m_{+}$ and  $\boldsymbol{d}=(d^{(1)},\ldots,d^{(m)})^{\rm T}  \in
  \mathbb{R}^m_{+}$. The loss functions
  \begin{equation}
 {L_k^{(m)} \left(\boldsymbol{\theta},\boldsymbol{d} \right) =  \sum_{j=1}^m \left( \left(\frac{d^{(j)}}{\theta^{(j)}} \right)^k - \left(\frac{\theta^{(j)}}{d^{(j)}} \right)^k \right) - 2m}
\label{multiloss}
\end{equation}
\begin{equation}
L_{sq}^{(m)} \left(\boldsymbol{\theta},\boldsymbol{d} \right) =  \sum_{j=1}^m \frac{(d^{(j)}-\theta^{(j)})^2}{d^{(j)}}
\label{multiloss2}
\end{equation}
are additive multivariate generalizations of the loss function $L_1$ given in (\ref{mauroloss}) and (\ref{lossscale}) respectively, which infinitely penalize each boundary decision $d^{(j)}=0$ and $d^{(j)}=\infty $, $j=1,\ldots,m$.
\end{proposition}
Clearly, loss functions $L_1^{(m)}$ and $L_{sq}^{(m)}$ penalize the boundaries as desired and some good performance of the corresponding estimators can be expected. However, the property of symmetry is not satisfied. One can find two decisions $\tilde{\boldsymbol{d_1}}$ and $\tilde{\boldsymbol{d_2}}$ for which $\mathcal{D}_{+}^{(m)} \left( \boldsymbol{\theta}, \tilde{\boldsymbol{d_1}} \right) = \mathcal{D}_{+}^{(m)} \left( \boldsymbol{\theta}, \tilde{\boldsymbol{d_2}}\right)$, but $L_1^{(m)} \left(\boldsymbol{\theta},\tilde{\boldsymbol{d_1}}  \right) \neq L_1^{(m)} \left(\boldsymbol{\theta},\tilde{\boldsymbol{d_2}}  \right)$.
Even if the whole loss functions are not symmetric, they are ``component-wise'' symmetric as shown above. The comparison of loss functions $L_1^{(2)}$, $L_{sq}^{(2)}$  and $\mathcal{D}^{(2)}_{+}$ for different values of decision $\boldsymbol{d_1}$ and $\boldsymbol{d_2}$ and fixed $\boldsymbol{\theta} = (1,2)^{\rm T}$ is given in Figure~\ref{fig:multi}.

\begin{figure}[ht!] 
\centering
\includegraphics[width=1\textwidth]{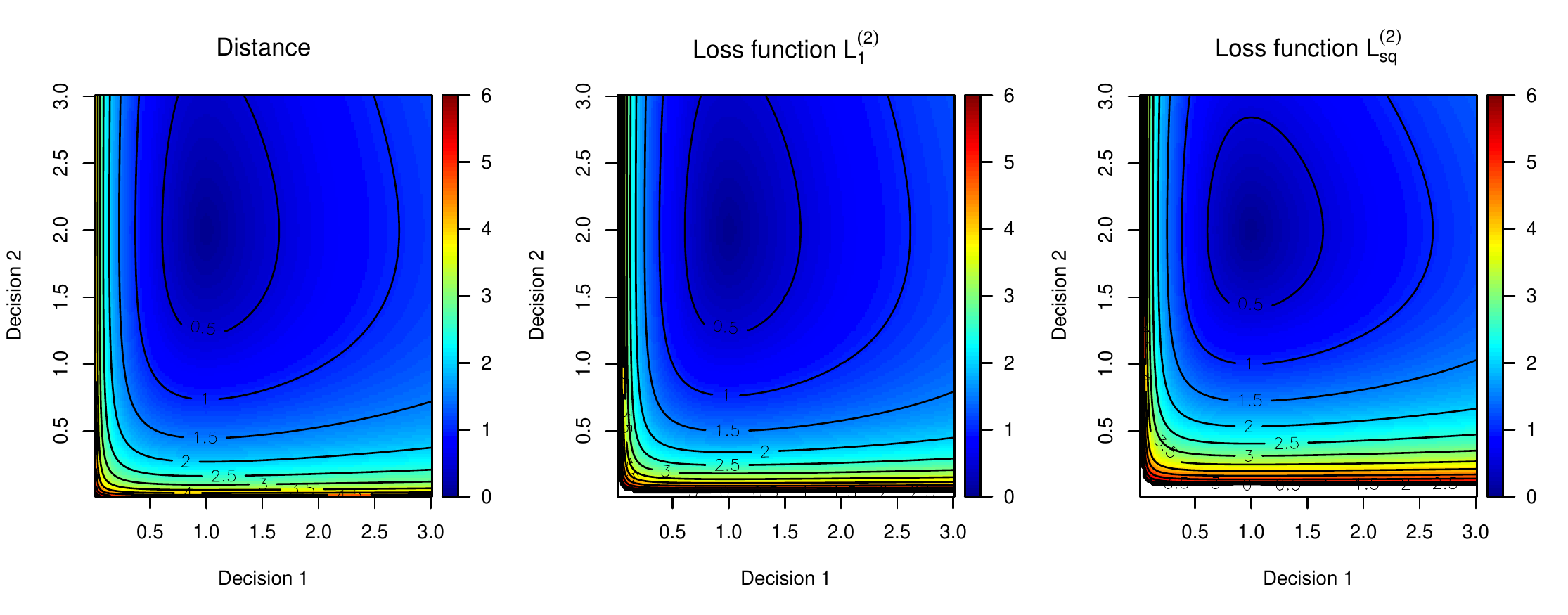}
 \caption{{Contour plots of loss functions $\mathcal{D}_{+}^{(2)}$, $L_1^{(2)}$,  $L_{sq}^{(2)}$ for the case $m=2$ and $\boldsymbol{\theta} = (1,2)^{\rm T}$}. 
\label{fig:multi}}
\end{figure}

The proposed loss functions perform similar to the distance $\mathcal{D}^{(2)}_{+}$, but have some more favourable properties, like convexity. 

\section{Examples}

Loss functions penalising boundary decisions were found to be more beneficial in many applications areas. For instance, \cite{sainthilary2018} have shown that a loss function similar to the loss function~(\ref{mauroloss}) penalising the decision $d=0$ can lead to a more reliable benefit-risk analysis of novel drugs. Similarly, \cite{mozgunov2018crm} have found that the loss function~(\ref{intervalloss}) penalising decisions $d=0,1$ incorporated into a model-based dose-finding design can improve a selection of the optimal doses without exposing patients to excessively toxic doses. Below, we investigate the performance of the proposed loss function and corresponding Bayes estimators in more general settings. We consider {four} classic examples of estimation to demonstrate the essential differences of estimators. We focus on the small sample size ($n=15$) to emphasize the difference in estimators. The results for moderate $(n=100)$ and large $(n=1000)$ sample sizes are given in Supplementary Materials. 

{For all examples,} the frequentist operating characteristic, Mean Squared Error (MSE), is chosen to compare the different estimators on common grounds. {As advocated by~\cite{berger2006case, efron2015frequentist}, studying the frequentist properties of Bayes estimators is a way to study the properties independently of the prior distribution and to consider Bayesian point estimate simply as a function of the data. Furthermore, as we intend to compare several Bayes estimators, which minimise different loss functions, the frequentist characteristics are chosen to assess the performance of these estimators on an equal basis.} Note that this choice is not favourable to our new proposals, since the MSE is derived from squared error loss.

\subsection{Estimation of a Probability \label{sec:probability}}

An important example of a parameter defined on the  finite interval $[0,1]$ 
is a probability. In the presence of a binary random sample with an unknown
probability of success a uniform distribution, i.e. a Beta
prior distribution $\mathcal{B}(1,1)$ 
is often assumed, a proposal which dates back to Laplace. Having observed
$x$ successes out of $n$ trials,
the posterior distribution is a conjugate Beta distribution $\mathcal{B}(x+1,n-x+1)$.
The estimator corresponding to the squared error loss function  (posterior mean) has the form
$
\hat{p}_q=\frac{x+1}{n+2}.
$
Another widely used estimator
 is the so-called ``add two successes and two failures'' 
Agresti-Coull estimator~\citep{agresti}
$\hat{p}_{AC}=\frac{x+2}{n+4}.$ Below we compare these approaches to the newly proposed estimator~(\ref{intervalestimator}). 

The symmetric optimal Bayes estimator~(\ref{intervalestimator}) in the
case $a=0$ and $b=1$ can be written
$$
\hat{d}_{iq}=\frac{\mathbb{E}(\theta^2)-\sqrt{\mathbb{E}(\theta^2)(1-2\mathbb{E}(\theta)+\mathbb{E}(\theta^2))}}{2\mathbb{E}(\theta)-1}$$
where $\theta$ is a probability of success lying
in the interval (0,1), over which a posterior distribution 
is given. It is assumed that the extremes of the interval are
not possible values for the parameter. The first and second moments of a Beta
distribution can be computed explicitly and plugged in formula (\ref{intervalestimator}) to obtain the following interval
symmetric optimal Bayes estimator
\begin{equation}
\hat{p}_{iq}=\left(1+\sqrt{\frac{(n-x+1)(n-x+2)}{(x+1)(x+2)}}\right)^{-1}.
\label{betaestimator}
\end{equation}

Simulated trials with sample size $n=15$ are considered. {On a grid of values $\theta \in (0.01,0.99)$, $N=10^9$ trials were simulated. This means that for each value of $\theta$ on the grid, we simulate $10^9$ trials with the total sample size $n=15$. This results in $10^9$ point estimates found for each method. Then, the MSE is computed as} 
\begin{equation}
MSE_k \equiv \frac{1}{N}\sum_{i=1}^N(\hat{p}_k^{(i)}-\theta)^2,
\label{eq:mse}
\end{equation}
where $\hat{p}_k^{(i)}$ is a corresponding value in the
$i^{th}$ simulation and $k=q,iq,AC$ corresponds to an estimation method.  
%
%
The results are given in Figure \ref{comparisonbeta}.

\begin{figure}[ht!]
  \centering
      \includegraphics[width=1\textwidth]{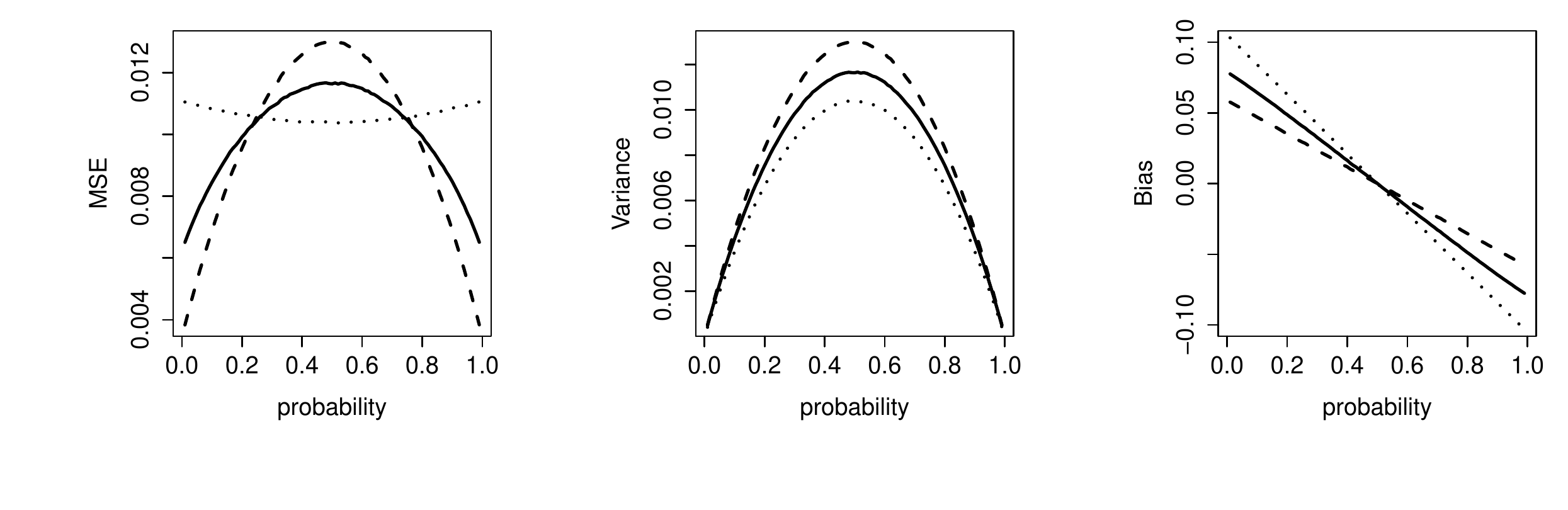}
       \caption{{MSE, variance and bias for the restricted symmetric squared error loss function estimator $\hat{p}_{iq}$ (solid), the squared error loss function estimator $\hat{p}_q$ (dashed) and the Agresti-Coull estimator $\hat{p}_{AC}$ (dotted). Results are based on $n=15$ observations and $10^9$ simulations.} \label{comparisonbeta}}
     \end{figure}
     
 The proposed estimator $\hat{p}_{iq}$ outperforms (in terms of the
     MSE) the Bayes estimator obtained using the squared error loss
     function $\hat{p}_q$ in the interval $\theta \in (0.2,0.8)$. The cost
     of this advantage is the worse performance on the intervals close
     to the bounds as the proposed form of the loss function penalizes
     boundary decisions and by that drives the final estimate away from them. However, the proposed estimator outperforms the
     Agresti-Coull estimator $\hat{p}_{AC}$ at the same intervals $\theta \in
     (0,0.2)$ and $\theta \in (0.8,1)$. Thus, the proposed estimator might
     be considered as a trade-off between currently used estimators
     $\hat{p}_q$ and $\hat{p}_{AC}$, that outperforms $\hat{p}_{AC}$ on
     bounds and $\hat{p}_q$ away from bounds.  
     
 In addition to the MSE,
the associated confidence intervals and coverage probabilities are extensively studied in the literature \citep[see e.g.][]{brown2}. 
In particular, coverage probabilities were shown to have an erratic behaviour
     and often to go below their nominal level. 
Corrections were proposed by \cite{agresti}. 
Confidence intervals can also be constructed around our newly proposed point estimator $\hat{p}_{iq}$.
The following confidence intervals are compared via simulated coverage probabilities in Figure~\ref{fig:coverage}:

\begin{enumerate}[label={[\arabic*]}]
\item  Normal approximation confidence interval centred around $\hat{p}_k$, $k=q,iq,AC$ as suggested by~\cite{brown2}
$$CI_{N}^{(k)} = \hat{p}_k \pm z_{\frac{\alpha}{2}} \sqrt{\frac{\hat{p}_k(1-\hat{p}_k)}{n}},$$
where $1-\alpha$ is  the confidence level.

\item Wilson confidence interval centred around $\hat{p}_{AC}$  \citep{wilson}
$$CI_{W}^{(AC)} = \frac{x+2}{n+4} \pm 2 \frac{\sqrt{n}}{n+2} \sqrt{\frac{x(n-x)}{n^2}+ \frac{1}{n}}.$$

\item  Approximate confidence interval using the delta-method centred around the newly proposed $\hat{p}_{iq}$ 
$$CI^{(iq)}_{D} = \hat{p}_{iq} \pm 2 \sqrt{\hat{V}_{iq}}, \ \hat{V}_{iq} = \left( \frac{ \partial f(x)}{ \partial x} \right)^2 {\Big
  |}_{x=n\hat{p}_{iq}}n\hat{p}_{iq}(1-\hat{p}_{iq}), \ 
f(x)=\hat{p}_{iq}.$$
\end{enumerate}

  \begin{figure}[ht!]
  \centering
       \includegraphics[width=1\textwidth]{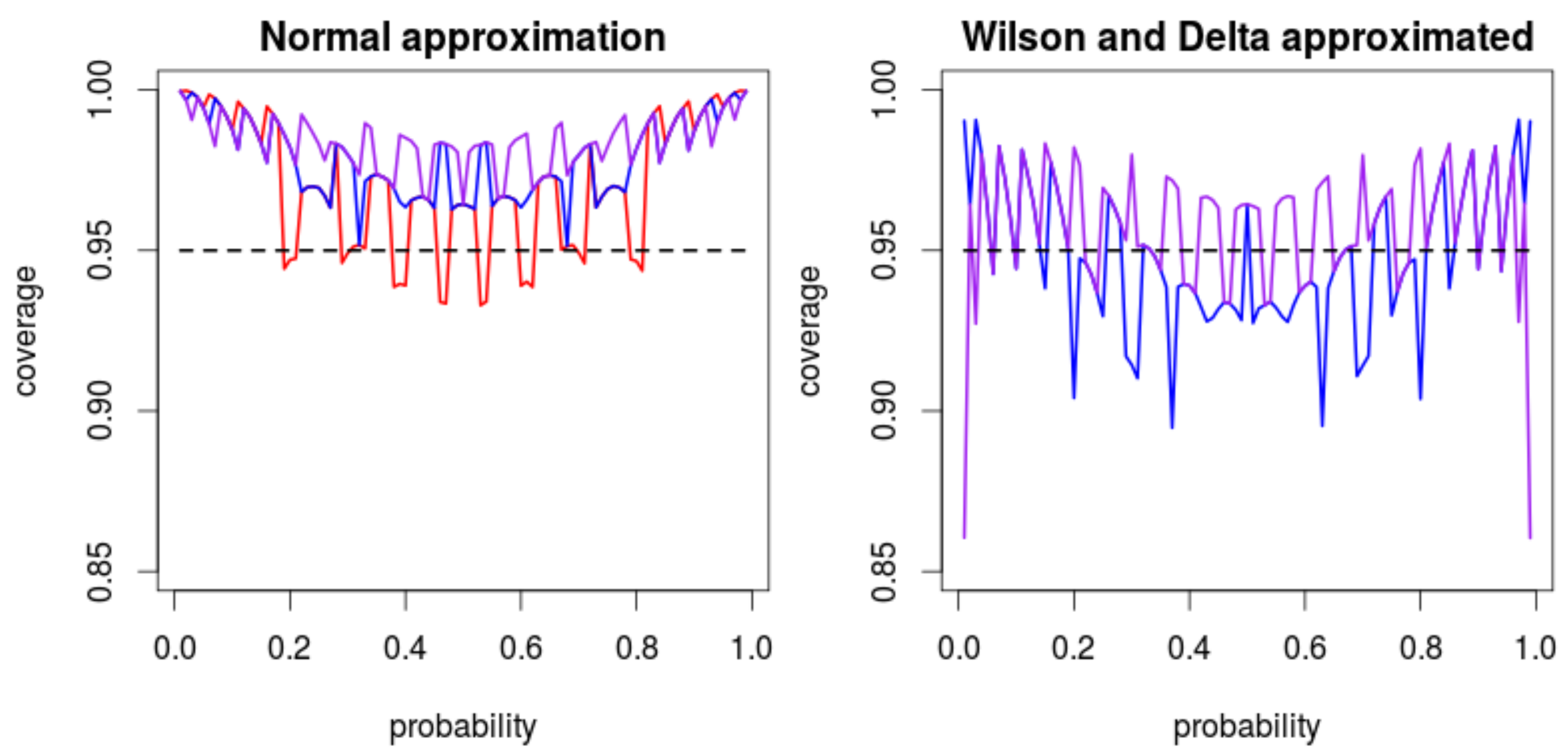}
       \caption{{Left panel: Coverage probabilities of $CI_{N}^{(iq)}$ (blue line), $CI_{N}^{(q)}$ (red line) and  $CI_{N}^{(AC)}$ (purple line) using  [1], the normal approximation interval. Right panel: Coverage probabilities of $CI^{(iq)}_{D}$ (blue line) and $CI_{W}^{(AC)}$ (purple line)  using [2], the Wilson, and [3], the delta-method, confidence intervals respectively. Results are based on $n=15$ observations and $10^9$ simulations.} \label{fig:coverage}}
     \end{figure}
     
 Using the normal approximation confidence interval, the coverage probability of $CI_{N}^{(q)}$ goes below the nominal value for several values of $\theta$. The coverage probabilities of $CI_{N}^{(iq)}$ and    $CI_{N}^{(AC)}$ also fluctuate but do not go below $0.95$ for $N=15$,
     which is a desirable property. While one can find combination of $N$ and $\theta$ for which coverage probabilities $CI_{N}^{(iq)}$ and $CI_{N}^{(AC)}$ might be below $0.95$ \citep{brown2}, it would be generally true that their coverage probabilities are greater than of $CI_{N}^{(q)}$ for larger intervals of $\theta$ and are more robust. In addition, a comparison between the normal approximation method (left panel of Figure~\ref{fig:coverage}) and the Wilson and delta method intervals
(right panel of Figure~\ref{fig:coverage}) supports the suggestion by~\cite{brown2} that the normal approximation method gives a portmanteau way to construct simple confidence intervals
with - on average - better coverage probabilities
than more complicated methods. 

\subsection{Restricted Estimation of a Normal Distribution Mean}

In the following example, it is demonstrated what benefits the proposed form of 
loss function~(\ref{intervalloss}) can provide in a Bayesian framework in the
presence of the additional information that the true parameter lies in
an interval $(a,b)$. 

The problem of restricted mean estimation has been known for a long time and has been extensively studied in the literature~\citep[see e.g.][and references therein]{kumar}. The previously proposed estimators were constructed using the squared error loss function and compared by the Bayesian risk. For an extensive overview of the problem we refer the reader to \cite{restricted} and for some recent generalizations to \cite{marchand}. Interestingly, despite the variety of the literature on the problem and the fact that the Bayes estimator with respect to the uniform prior
 distribution on $(-a,a)$ outperforms uniformly the ``unrestricted'' Bayes estimator under squared error loss function~\citep{hartigan}, the sample mean estimator is still widely used in practice. For this reason, we propose our simple alternative and compare it to the most commonly used estimators. 

Consider the problem of estimating the mean  $\mu$ of a
normal sample of i.i.d. $X_i \sim \mathcal{N}(\mu,\sigma^2)$, $i=1,\ldots,n$
where $\sigma^2$ is known. Assume  it is known that the
parameter $\mu$ belongs to the interval $(-a,+a)$ where $a >0$. A possible example is the estimation of the treatment
effect in paediatric studies of a clinical treatment that was already tested on
adults. An investigator might be sure that the same dosage of the drug as given to adults would cause a greater level of toxicity in children. 
It is of interest how much one can
gain by incorporating this information in the estimator~(\ref{intervalestimator}) compared to currently used approaches that
use the squared error loss. 

As stated above, the common way to incorporate this information in a Bayesian framework
is restricting the prior distribution for the parameter of interest $\mu$ to the
interval $(-a,+a)$ and then using the squared error loss to obtain
a point estimate (as a summary of a posterior distribution). However,
the information about the restricted space  is used only on
the prior and ignored when choosing summary statistics,
 while the proposed form of loss function~(\ref{intervalloss})
allows to incorporate it again. Moreover, in practice the prior information is often ignored and the sample mean estimator (corresponding to Jeffrey's prior) is used. Therefore, a comparison of the proposed loss function and
the currently utilized approaches (with and without the incorporation
the prior information) is of interest. 

Consider samples of size  $n=15$ and alternative Bayes estimators as follows:

\begin{itemize}
\item Bayes estimator under Jeffrey's prior  $g(\mu) \propto k \sqrt{\frac{n}{\sigma^2}}$ for $\mu$ and the squared error loss function. Denoted by $J$;

\item Bayes estimator under the uniform prior $\mathbb{U} \sim (-a,+a)$ for $\mu$ 
and squared error loss function. Denoted by $U_1$.

\item Bayes estimator~(\ref{intervalestimator}) under the uniform prior $\mathbb{U}(-a,+a)$ for $\mu$ and the interval symmetric squared error loss function. Denoted by $U_2$.
\end{itemize}

Alternatively to $U_2$, the Bayes estimator $U_2^{\prime}$ defined by~(\ref{intervalestimator}), but with a wider interval $(-1.25a,+1.25a)$ is also applied to investigate how a lower penalization of the bounds influences the estimation. This wider interval could be considered as a conservative way to incorporate information about the location of the parameter. 

Two cases $a=2$ and $a=4$ are considered. The parameter $\sigma^2=4$ is assumed to be known in both cases. As before, the $MSE$ is used for a fair comparison of methods.  The results for $10^6$ replications for each value of $\mu$ are given in Figure \ref{fig:normal}. 

  \begin{figure}[ht!]
  \centering
       \includegraphics[width=1\textwidth]{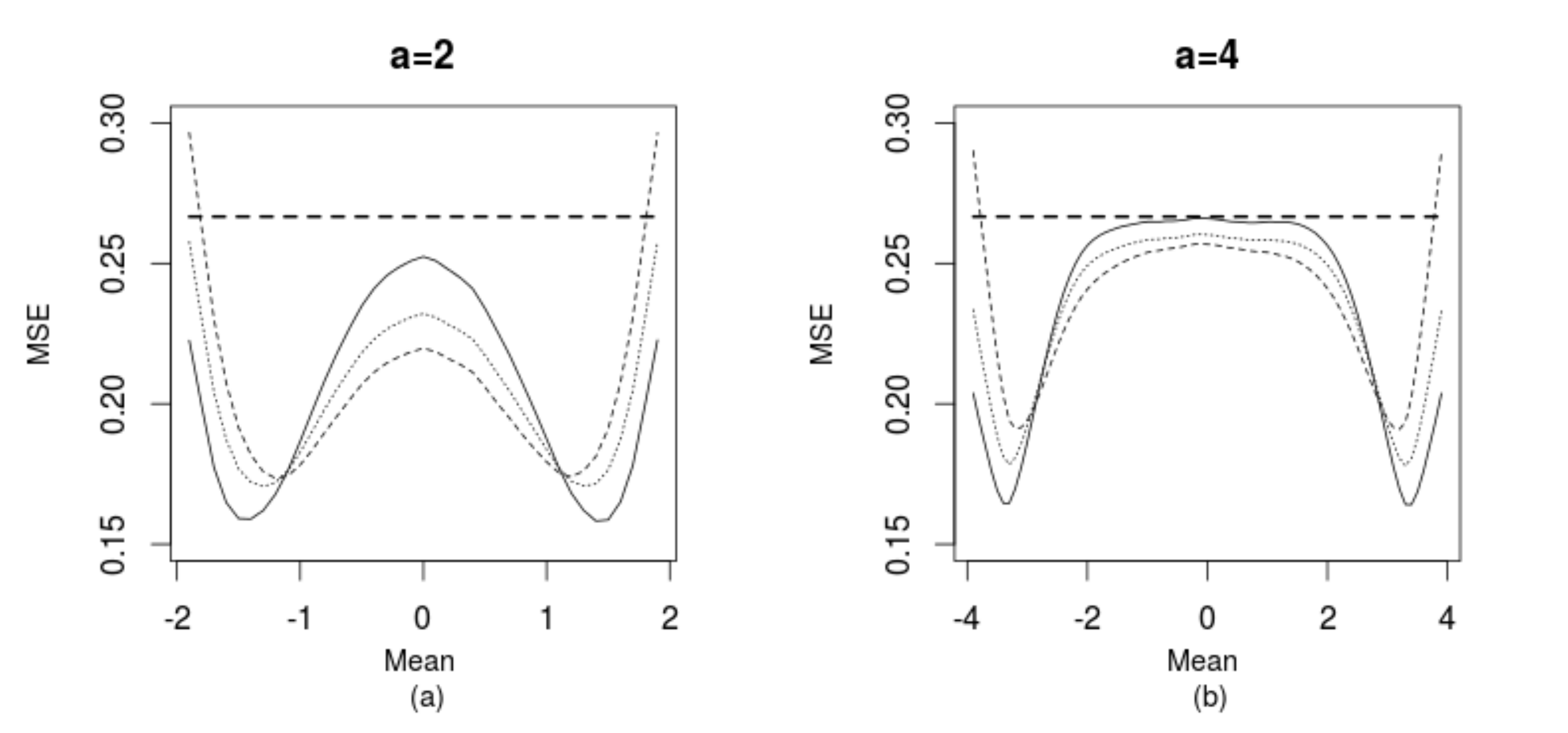}
       \caption{MSE corresponding to different values of the restricted mean parameter $\mu$ with (a) $a=2$ and (b) $a=4$ and  the Bayes estimator $U_1$ (solid), $U_2$ (dashed) and $U_2^{\prime}$ (dotted) and simple mean estimator  $J$ (solid dashed). Results are based on $10^6$ replications. \label{fig:normal}}
     \end{figure}
     
Incorporating the interval information in the Bayes estimator allows for improvement if the true value of the parameter $\mu$ does not lie close to the bound. In the case $a=2$ estimator $U_2$ outperforms $U_1$ on the interval $\mu \in (-1,+1)$ and in the case $a=4$ on the interval $\mu \in (-3,+3)$. The same holds for estimator $U_2^{\prime}$, however, its MSE never falls below the level of the MSE of estimator $J$. Clearly, the wider interval improves estimation on the bounds, at the cost of the higher MSE in the middle. Note that the wider interval $a=4$ corresponds to weaker additional information and to a smaller benefit in the MSE comparing $U_1, U_2$ and~$U_2^{\prime}$ against $J$. Overall, the Bayes estimator corresponding to the interval symmetric loss function avoids the boundary decisions, improves the estimation if the parameters lies away from bounds and can be recommended for the application if boundary decisions lead to severe consequences.

\subsection{Bayesian Estimation of the Parameters of a Gamma Distribution}
An important example of  multidimensional restricted parameter estimation is a Gamma distribution with positive shape and scale parameters $\alpha_1,\alpha_2 >0$. Bayesian inference for this problem has been studied, for example, in \cite{miller} and methods for approximate computation of Bayes estimators were proposed using the Lindley approximation~\citep{lind}. 

We consider the function $L_{sq}^{(2)}$ given in (\ref{multiloss2}) to obtain Bayes estimators. As the loss function (\ref{multiloss2}) is the sum of the univariate precautionary loss functions~(\ref{lossscale}), the following estimators are  used $
\hat{\alpha}_i= \sqrt{\mathbb{E}(\alpha_i^2)}, \ i=1,2
$
where the expectation $\mathbb{E}$ is taken with respect to the posterior distribution. 

Let us consider an experiment with sample size $n=15$. The parameters of the Gamma distribution are varied over a grid, $\alpha_1,\alpha_2 \in (0,10) $ and the performance of the different approaches are compared by simulations. The two approaches compared use the same prior distribution for parameters $\alpha_1$ and $\alpha_2$ using the following estimators:
 \begin{enumerate}[label={\arabic*)}]

\item Bayes estimator under the squared error loss function,

\item Bayes estimator under the multivariate precautionary loss function.
\end{enumerate}

{We use the same weakly informative Gamma prior distributions $\Gamma(10^{-4},10^{-4}$) of ($\alpha_1, \alpha_2)$ with parameters for both estimation methods} and an approximate method proposed by \cite{lind}. {The weakly informative distribution are chosen to minimise the influence of the prior distribution on the comparison of estimators and corresponds to the situations of no prior knowledge about the parameters.} We would like to emphasize that the choice of the prior distribution is out of the scope of this paper and the goal is to compare two Bayes estimators when all other parameters are equal. The difference between the MSE of estimators based on method (1 against method 2) for parameters $\alpha_1$ and $\alpha_2$ in $10^9$ simulations is given in Figure~\ref{fig:gamma_example}.

     \begin{figure}[ht!]
  \centering
       \includegraphics[width=1\textwidth]{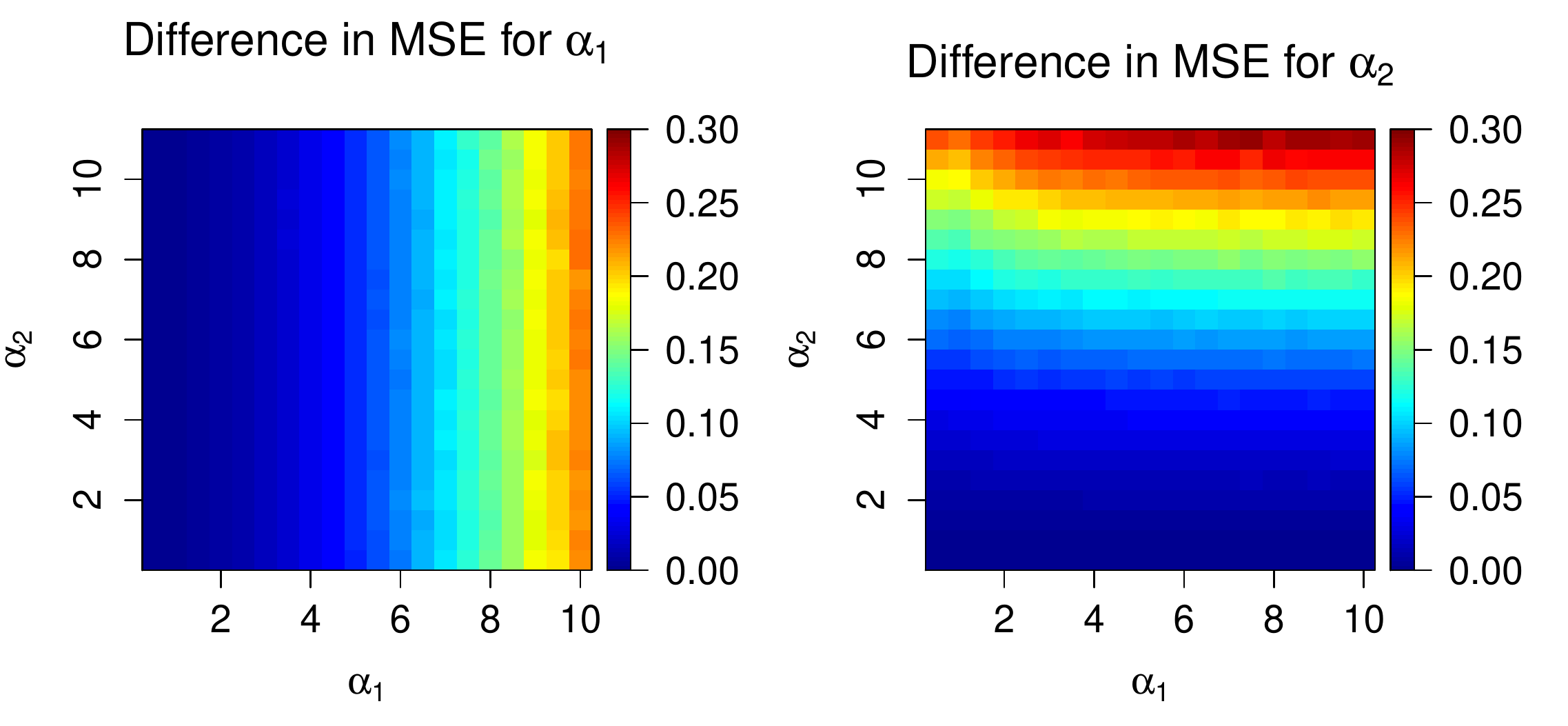}
       \caption{Difference in the MSEs for parameters $\alpha_1$ and $\alpha_2$ for their different true values and using Bayes estimator under the squared error loss function and Bayes estimator under $L_{sq}^{(2)}$. Results are based on $10^9$ replications.  \label{fig:gamma_example}}
     \end{figure}

The differences in the MSEs for both parameters are positive for all values of  the true parameters $\alpha_1$ and $\alpha_2$. It means that the Bayes estimator from method 2) is associated with smaller MSE than the Bayes estimator from method 1). The difference in the MSE increases as the true value of the parameter increases. This result makes the proposed estimator and the associated loss function $L_{sq}^{(2)}$ good candidates for further investigation in multidimensional estimation problems.

\subsection{{Bayesian Estimation of the Parameters of a Weibull Distribution}}
Another important example of multidimensional restricted parameters estimation is the Weibull distribution which positive scale and shape parameters $\lambda, \nu >0$. The Weibull distribution is of great importance in applications as it is widely employed in, for example, reliability engineering, extreme value theory and survival analysis~\citep{nwobi2014comparison}.  Bayesian inference for this problem has been studied in \cite{gupta2017classical}. Importantly, despite both parameters being defined on the positive real line only, the squared error loss function (and associated posterior mean Bayes estimator) are used in these works. Below, we consider how the novel loss function~(\ref{scaleinvariantloss}) and the associated Bayes estimator behaves in this estimation problem.

We consider the function $L_k^{(m)}$ given in~(\ref{multiloss}) to obtain Bayes estimators. As the loss function~(\ref{multiloss}) is the sum of the univariate loss functions~(\ref{mauroloss}), the following Bayesian estimators are used for both scale parameters of Weibull distribution
\begin{equation}
\hat{\lambda}_k=\left(\frac{\mathbb{E}(\lambda^k)}{\mathbb{E}(\lambda^{-k})}\right)^{1/2k}, \ \ \hat{\nu}_k=\left(\frac{\mathbb{E}(\nu^k)}{\mathbb{E}(\nu^{-k})}\right)^{1/2k}
\label{weibull}
\end{equation}
where the expectations are taken with respect to the posterior density function. Note that the estimators depend on the parameter of the loss function $k$. We will investigate the influence of the parameter on the estimation.

As above, we consider an experiments with a small sample size, $n=15$. We consider several values of both scale and shape parameters, $\lambda \in \{1,2,3,4,5\}$, $\nu \in \{(0.5,1,5,10,15)\}$, and the performance of different approaches are compared by simulations. Again, the two approaches compared use the same weakly information Gamma prior distribution $\Gamma(10^{-4},10^{-4})$ for both positive parameters, as we would like to minimise the impact of the prior distribution on the comparison.

We start from the comparison of 
 \begin{enumerate}[label={\arabic*)}]

\item Bayes estimators under the squared error loss function,

\item Bayes estimators~(\ref{weibull}) under the multivariate $L_1^{(m)}$~(\ref{multiloss}) for $k=1$.
\end{enumerate}
The difference between the MSE of estimators (method 1 against method 2) for positive parameters $\lambda$, $\nu$ in $10^4$ simulations are given in Table~\ref{tab:weibull}. The MSE for $\nu$ are scaled by $\frac{1}{\nu^\lambda}$ to obtain the results on a similar scale for various parameters.

\begin{table}[ht!]
\centering
       \caption{Difference in the MSEs for parameters $\lambda$ (upper lines) and $\nu$ (lower lines) for their different true values and using Bayes estimator under the squared error loss function and Bayes estimator under $L_1^{(m)}$. Results are based on $10^4$ replications.  \label{tab:weibull}}
\begin{tabular}{cccccc}
  \hline
 &  $\nu=1$ & $\nu=2$ &  $\nu=5$ & $\nu=10$ & $\nu=15$ \\ 
  \hline
\multirow{2}{*}{$\lambda=1$} &  0.0061 & 0.0065 & 0.0067 & 0.0069 & 0.0068 \\ 
 &  0.0097 & 0.0081 & 0.0033 & 0.0067 & 0.0087 \\
    \hline

\multirow{2}{*}{$\lambda=2$} &  0.0268 & 0.0270 &  0.0274 & 0.0265 & 0.0269 \\ 
&  0.0090  & 0.0106 & 0.0103 & 0.0142 &  0.0142 \\
   \hline

\multirow{2}{*}{$\lambda=3$} &  0.0608 &  0.0616 &  0.0598 & 0.0622 & 0.0613\\ 
 &  0.0081 & 0.0125  &  0.0143  & 0.0136 &  0.0147 \\
    \hline

\multirow{2}{*}{$\lambda=4$} &   0.1075 & 0.1068 & 0.1065 & 0.1065 &  0.0987\\ 
 &    0.0104 & 0.0122  &  0.0146  &  0.0136  &  0.0151 \\
    \hline

\multirow{2}{*}{$\lambda=5$} & 0.1652 & 0.1649 & 0.1645 & 0.1652 & 0.1585 \\ 
&  0.0093 &   0.0131 &   0.0140  &   0.0115   &  0.0152 \\
   \hline
\end{tabular}
\end{table}

The differences in the MSEs for both parameters are positive for all considered true values of $\lambda$ and $\nu$. It follows that the Bayes estimator from method 2) leads to a smaller MSE than the estimator corresponding to the squared error loss function. For a fixed value of $\nu$, the MSE corresponding to $\lambda$ increases with the parameters. The scaled MSE corresponding to $\nu$ stays nearly the same for various values of $\lambda$. Overall, the proposed estimator and the associated loss function can be good candidates to be used for the parameters defined on the positive real line.

So far, only the loss function  $L_k^{(m)}$ for $k=1$ has been considered. To investigate the impact of the parameter $k$ on the estimation characteristics, we fix the scale parameter $\nu=1$ and vary the shape parameter over the grid $\lambda \in (1,8)$. Specifically, we consider the following value of the parameter $k=\{1,2,3,4\}$.  The results are given in Table~\ref{tab:kdependence}.

The posterior mean estimator (corresponding to the squared error loss function) corresponds to the highest MSE for all considered values of $\lambda$. As it was shown above, one can reduce the MSE by applying the estimator~(\ref{weibull}) for $k=1$. Considering the bias $k=1$, the estimator overestimates the true value of the parameter. When increasing the value of $k$, the value of the estimator decreases and approaches the true value of the parameter. This results in reduce bias of the estimator for all values of $\lambda$ as $k$ increases. Furthermore, the variance of the estimator decreases with parameter $k$. This results in decrease of the MSE as $k$ increases. Overall, increasing values of $k$ are found to lead lower estimates, which in the case of the overestimations leads to smaller MSE and more accurate estimation.

\begin{table}[ht!]
\centering
       \caption{MSE, Bias and Variance for the Bayes estimator of $\lambda\in (1,8)$ corresponding to the squared error loss function (posterior mean) and for the estimator $\hat{\lambda}_k$ given in Equation~(\ref{weibull}) for $k=1,2,3,4$. Results are based on $10^4$ replications.  \label{tab:kdependence}}
\begin{tabular}{cccccccccccc}
  \hline
 &  & $\lambda=1$ & $\lambda=2$ &  $\lambda=3$ & $\lambda=4$ & $\lambda=5$ & $\lambda=6$ & $\lambda=7$& $\lambda=8$ \\ 
 
  \hline
  \multirow{1}{*}{Posterior} & MSE & 0.07 & 0.28 & 0.65 & 1.15 & 1.78 &  2.58 & 3.51 &  4.57 \\
Mean & Bias & 0.09 &0.18 &0.28 &0.38 & 0.47 & 0.57 & 0.66&  0.77\\
 &Var  & 0.06 & 0.25 & 0.57 & 1.00 & 1.55 & 2.24&  3.06 & 3.97\\
   \hline

 \multirow{1}{*}{Est.~(\ref{weibull})} & MSE & 0.06 & 0.25 & 0.59 & 1.04 & 1.61 & 2.33 & 3.18&  4.13  \\
$k=1$  & Bias & 0.07 & 0.14 & 0.21 & 0.28 & 0.36 &  0.43&  0.50 & 0.58\\
 &Var  & 0.06 & 0.23 & 0.54 & 0.95 & 1.48 & 2.15&  2.93 & 3.80 \\
   \hline

 \multirow{1}{*}{Est.~(\ref{weibull})} & MSE & 0.06 & 0.25 & 0.58 & 1.03 & 1.60 & 2.32 & 3.16 & 4.11\\
$k=2$ & Bias & 0.07 & 0.13 & 0.21 &  0.21 & 0.34&  0.42&  0.49 & 0.56\\
 &Var  & 0.06 & 0.23 & 0.54 & 0.95 &  1.48 & 2.14 & 2.92 & 3.78\\
   \hline

   \multirow{1}{*}{Est.~(\ref{weibull})} & MSE & 0.06 & 0.25 & 0.58 & 1.06 & 1.58 & 2.32 & 3.15 &  4.02 \\
$k=3$ & Bias & 0.06 &0.13 & 0.20 & 0.27& 0.33 & 0.40 & 0.44 & 0.54\\
 &Var  & 0.05 &  0.23 & 0.54 & 0.95 & 1.45&  2.14 & 2.90 & 3.77\\
   \hline

     \multirow{1}{*}{Est.~(\ref{weibull})} & MSE & 0.06 &  0.25 & 0.57 & 1.02 & 1.56&  2.27 & 3.09 &  4.01\\
$k=4$ & Bias & 0.06 & 0.12 & 0.19&  0.25 & 0.31 & 0.38 & 0.44&  0.52\\
 &Var & 0.05 & 0.23 & 0.53 & 0.94&  1.46&  2.12 & 2.89 & 3.74\\
   \hline

\end{tabular}
\end{table}

\section{{Application of the novel estimator to the paediatric clinical trial sample size calculation}}

\subsection{Motivation}

There is overwhelming medical evidence  that children's response can noticeably differ from the adults' response in many diseases~\citep{stephenson2005children,klassen2008children}. While there is generally a large amount of data available from adult studies, the knowledge about children remains quite limited. An example is drug-resistant partial epilepsy. Despite guidelines establishing the need to perform comprehensive paediatric drug development programmes, pivotal trials in children with epilepsy have been completed mostly in Phase IV as a postapproval replication of adult data~\citep{rheims2008greater}. To change this practice, more studies in children should be conducted. The planning of such studies, however, requires many assumptions, and the information from adults population can (and should) be efficiently used to justify them. Specifically, the planning requires the values of the expected responses (given previous trials) for an alternative treatment (or control) to which the comparison is to be made. The underestimation of this expected value can lead to a underpowered study and to unethical allocation of children in the study. In this section, we will demonstrate how the information about the relation between adults' and children's response can be incorporated in the proposed interval symmetric estimator and how it impacts the subsequent planning of the clinical trial.

\subsection{Setting}
Assume that one would like to conduct a randomised controlled clinical trial to study whether a novel intervention leads to a significantly different response in children with inadequately controlled partial seizures. A typical question that a clinician would ask a statistician before the trial is the sample size required to achieve a desirable level of statistical power. Suppose that equal sample sizes for the intervention and control (placebo) groups are to be used. Formally, the clinician would like to test the following hypothesis
$$H_0:p \leq p_{placebo} \ {\rm versus} \ H_1:p > p_{placebo}$$
where $p$ is the known probability of response  for the tested intervention and  $p_{placebo}$ is the probability of response given the placebo. To achieve $1-\beta$ in testing this hypothesis and given the type-I error $\alpha$, one can obtain the following formula for the sample size
$$n \approx \frac{\left( z_{\alpha} \sqrt{2\bar{p}(1-\bar{p})} + z_{\beta} \sqrt{p_{target}(1-p_{target}) + p_{placebo}(1-p_{placebo})}\right)^2}{\left(p_{target} - p_{placebo}\right)^2} $$
where $p_{target}$ is the clinically important response which the clinician would like to find in the trial, and $z_\alpha$ and $z_\beta$ are $1-\alpha$ and $1-\beta$ quantiles of the normal distribution. The latter two values are to be defined by clinicians and statisticians. The clinically important response, given the severity of the diseases and alternative treatments~\citep{rheims2008greater}, is assumed to be $p_{target}=0.5$. However, the response corresponding to the placebo effect is much more challenging.

Given a vast amount of data from clinical trials in adults, it is known that the response to placebo is $0.10$~\citep{rheims2008greater}. However, one cannot use this value as there is an evidence that the placebo response can be lower for children. Moreover, the meta-analysis by~\citep{rheims2008greater} suggested the placebo response is at $0.20$ twice as large. Using this knowledge and the clinical data from the recent study~\citep{glauser2006double}, one can estimate the response for the placebo effect. We argue that the estimator corresponding to the interval symmetric loss function~(\ref{intervalloss}) is appropriate choice for at least two reasons:
\begin{enumerate}
\item It is known that the clinically feasible values of the placebo response start at $0.1$. Therefore, the probability of interest, $p_{placebo}$ lies in the interval $(0.1,1)$.
\item Underestimation of the placebo response is highly undesirable given that the trial is conducted in children. The underestimation leads to underpowered study, which might result in its failure and in the unethical ``waste'' of children patients involved.
\end{enumerate}
At the same time, one would like to limit a number of children enrolled in the study and the proposed sample size should be justified. The difference in the sample size calculations using the currently used ``naive'' estimator for the placebo and the proposed interval symmetric estimator~(\ref{intervalestimator}) is given below.

\subsection{Results}
Given the data obtained in the randomised clinical trial~\citep{glauser2006double}, there were $97$ children patients assigned to the placebo group and $19$ of them experienced a reduction of partial seizure frequency. Therefore, the ``naive'' estimator that would be typically used to plan the clinical trial is
$$\hat{p}_{naive}=\frac{19}{97}\approx 0.196.$$

Alternatively, using the information that the estimator $p\in(0.1,1)$ and assuming the uniform prior distribution for the probability and the Beta posterior $\mathcal{B}(x+1,n-x+1)$ as in Section~\ref{sec:probability}, one can obtain the following formula for the interval symmetric estimator~(\ref{intervalestimator})
$$\hat{p}_{iq}=\frac{0.1- \frac{(x+1)(x+2)}{(n+3)(n+2)} + \sqrt{\left(\frac{(x+1)(x+2)}{(n+3)(n+2)}-0.1 \right)^2 - \left(1.1-2\frac{x+1}{n+2} \right) \left(2.2\frac{x+1}{n+2} - 1.1 \frac{(x+1)(x+2)}{(n+3)(n+2)} \right) }}{1.1-2\frac{x+1}{n+2}}.$$
Plugging-in the data from the study~\citep{glauser2006double}, $x=19$ and $n=97$, one can obtain 
$$\hat{p}_{iq}\approx 0.209.$$
While the difference in the estimates can look quite marginal, it indeed leads to a difference in the required sample size per treatment group. Using $\alpha=0.05$ type-I error, and the desirable power of $1-\beta=0.90$, one can obtain that the estimators lead to the following sample sizes: $n_{naive}= 41$ and $n_{iq}=45$. Consequently, the proposed estimator suggests to enrol 8 more patients into the study. While this might seem to result in a minor change in the total sample size, this justified increase in the sample size can avoid a failure of the study and might lead to a new, better intervention available for children suffering from epilepsy.

\section{Discussion}
The concept of a symmetric loss function in a restricted parameter space is introduced in this paper. Scale symmetric and interval symmetric loss functions which share desirable properties are provided. On the basis of {four examples}, we show that the corresponding Bayes estimators perform well when compared to other available estimators based on squared error loss, and improve the estimation if the parameter lies away from bounds. {Following the real life application example, it is found that the novel Bayes estimators allow avoiding boundary decisions that can be undesirable in paediatric clinical trials.} Consequently, the estimator can be recommended in other applications where more conservative estimates are preferable. 

{Overall, when choosing a loss function for the particular application, we argue that a statistician should answer two questions: (i) is there credible information that the parameter of interest is restricted to particular space and (ii) are there any values of the parameters should be avoided as they might lead to undesirable consequences. Answering these questions will guide whether, for example, the squared error loss function, the scale symmetric loss function or the interval symmetric (with specified interval) loss function should be used.} We would like to emphasize that we restrict our choice to some specific loss functions, mainly due to their simplicity and easy implementation. Alternative loss function sharing the stated properties can and should be considered. 

The proposed definitions were generalized for a subset of $\mathbb{R}^m$, where distances on restricted spaces could themselves be used as loss functions, which usually are non-convex and do not result in explicit minimizers. While we have presented the modification of the squared error loss function for  a restricted univariate parameter defined on an interval, the equivalent multidimensional extension seems to be non-trivial and requires further study. 

\section*{Acknowledgements}
The authors acknowledge the insightful and constructive comments made by associate editor and two reviewers. These comments have greatly helped to sharpen the original submission.

\section*{Funding}
This project has received funding from the European Union’s
Horizon 2020 research and innovation programme under the 
Marie Sklodowska-Curie grant agreement No 633567 and by the Medical Research Council (MR/M005755/1) and, in part, from Prof Jaki's Senior Research Fellowship (NIHR-SRF-2015-08-001) 
supported by the National Institute for Health Research. The views expressed in this publication are those of the authors and not necessarily those of the NHS, the National Institute for Health Research or the Department of Health and Social Care (DHCS)..

\bibliographystyle{tfs}
\bibliography{thesis}

\clearpage

\section*{Supplementary Materials}

\section*{Examples: Moderate and Large Sample Sizes}
While the cases of small sample sizes were used in the main body of the manuscript only $(n=15)$, it is also on interest to investigate how different the corresponding results for moderate and large sample sizes. Below, we revisit four examples considered in the work using moderate $(n=100)$ and large ($n=1000)$ sample sizes. Overall, all stated conclusions about the relative performances of the proposed estimators compared to the standard one stand. While the quantitatively an increased sample sizes leads to smaller differences (due to the improved estimation for all methods), the qualitative patterns are the same.

\subsection*{Estimation of a Probability}
Figure~\ref{probability100} presents the MSE, variance and bias for the proposed estimator of a probability, for the Agresti-Coull estimator, and for the squared error loss function estimator using the total sample size $n=100$.

\begin{figure}[!h]
  \centering
      \includegraphics[width=1\textwidth]{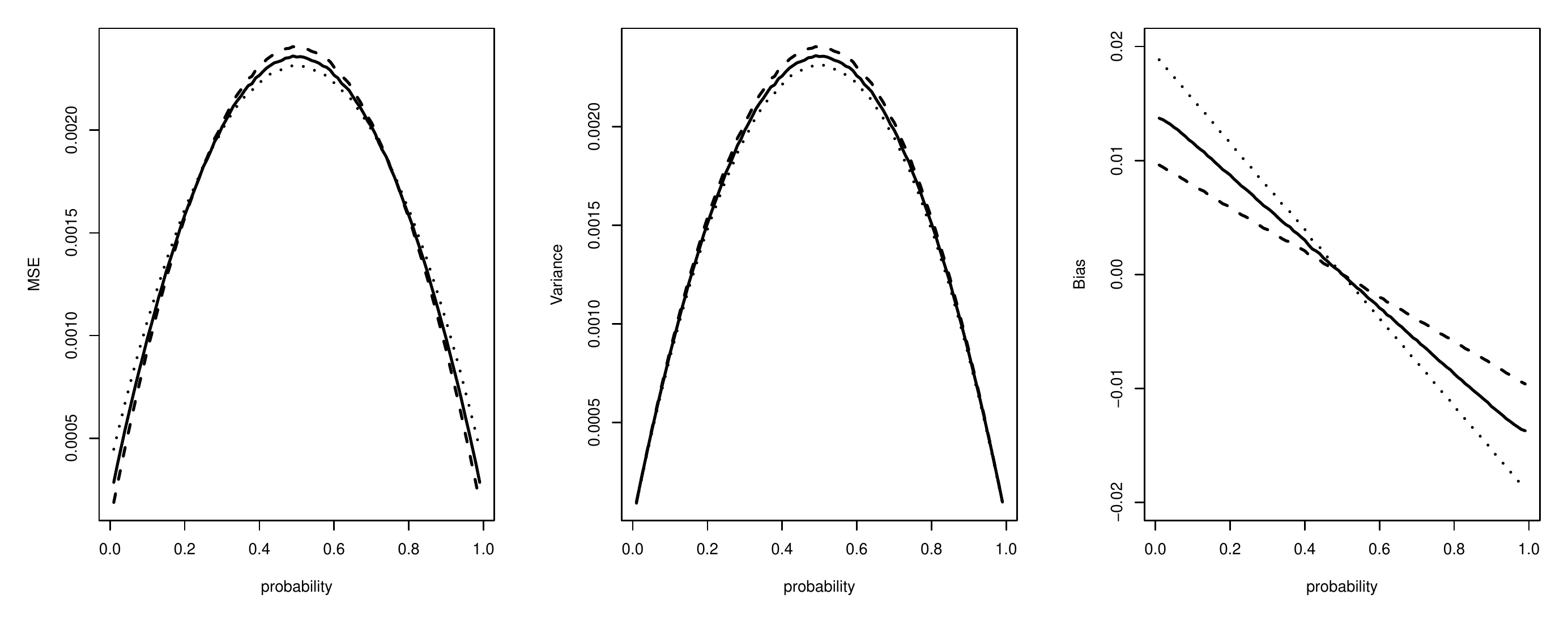}
       \caption{{MSE, variance and bias for the restricted symmetric squared error loss function estimator $\hat{p}_{iq}$ (solid), the squared error loss function estimator $\hat{p}_q$ (dashed) and the Agresti-Coull estimator $\hat{p}_{AC}$ (dotted). Results are based on $n=100$ observations and $10^4$ simulations.} \label{probability100}}
     \end{figure}
     
As expected, the magnitude of differences is now smaller for all considered characteristics. Specifically, the MSE associating with $\hat{p}_q$ is much closer to those of $\hat{p}_{AC}$ and $\hat{p}_{iq}$. Nevertheless, the proposed estimator $\hat{p}_{iq}$, again, performs better (in terms of the MSE)  than the Bayes estimator obtained using the squared error loss
     function $\hat{p}_q$ in the interval $\theta \in (0.2,0.8)$, and than Agresti-Coull estimator $\hat{p}_{AC}$ in the interval $\theta \in (0.8,1)$. 
     
Figure~\ref{probability1000} presents the MSE, variance and bias for the proposed estimator of a probability, for the Agresti-Coull estimator, and for the squared error loss function estimator using the total sample size $n=1000$.
  
\begin{figure}[!h]
  \centering
      \includegraphics[width=1\textwidth]{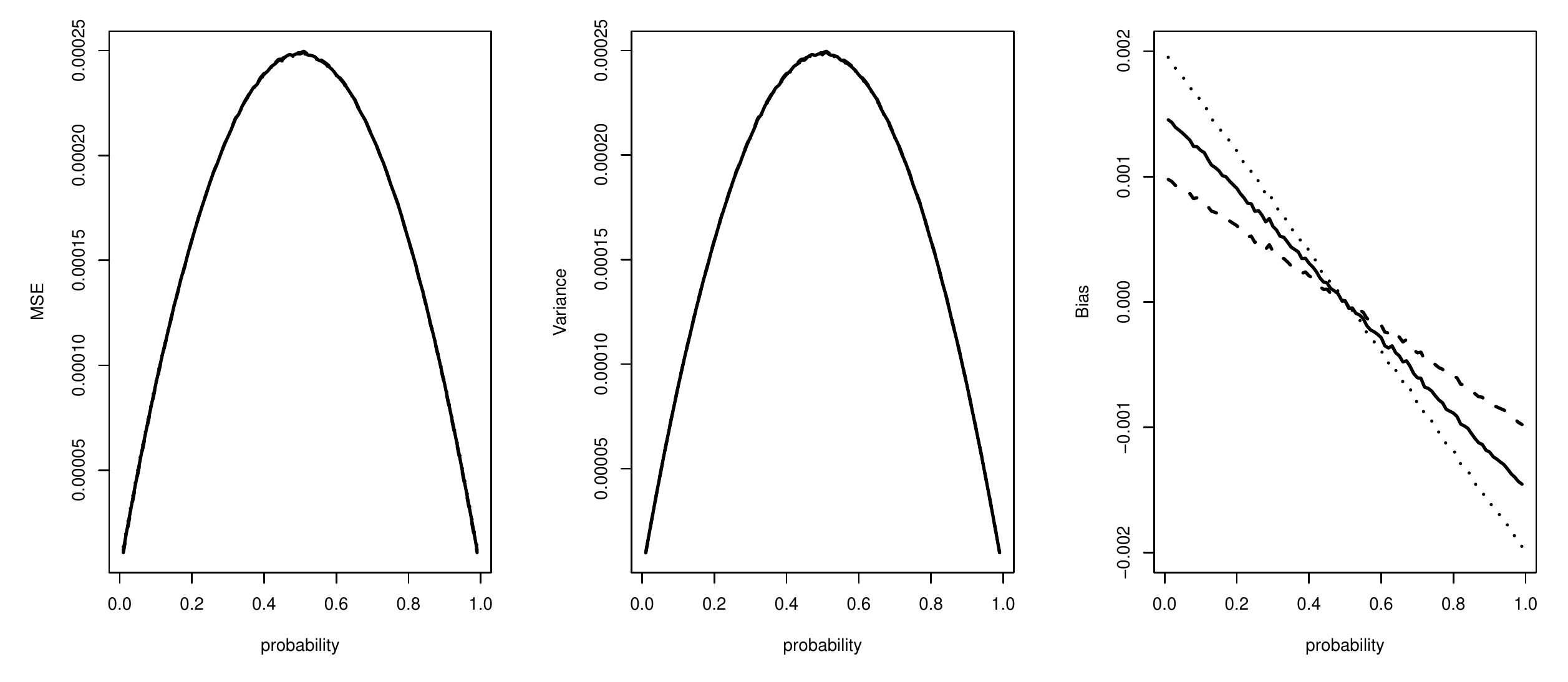}
       \caption{{MSE, variance and bias for the restricted symmetric squared error loss function estimator $\hat{p}_{iq}$ (solid), the squared error loss function estimator $\hat{p}_q$ (dashed) and the Agresti-Coull estimator $\hat{p}_{AC}$ (dotted). Results are based on $n=1000$ observations and $10^4$ simulations.} \label{probability1000}}
     \end{figure}

For even large sample sizes, all three estimators have same characteristics and the curves corresponds to the MSE and Variance are not distinguishable. Also, the differences in bias are of smaller magnitude. Overall, the proposed estimator does provide particular benefits in the intervals of $\theta$ for small and moderate sample sizes while performing similarly to alternatives for the large sample sizes.

     \subsection*{Restricted Estimation of a Normal Distribution Mean}
Figure~\ref{normal100} presents the MSE associated with the four considered estimators $J$, $U_1$, $U_2$,  $U_2^{\prime}$, for each value of $\mu$ in intervals (-a,a) using the total sample sizes $n=100$, and $a=2$ and $a=4$.

\begin{figure}[!h]
  \centering
      \includegraphics[width=1\textwidth]{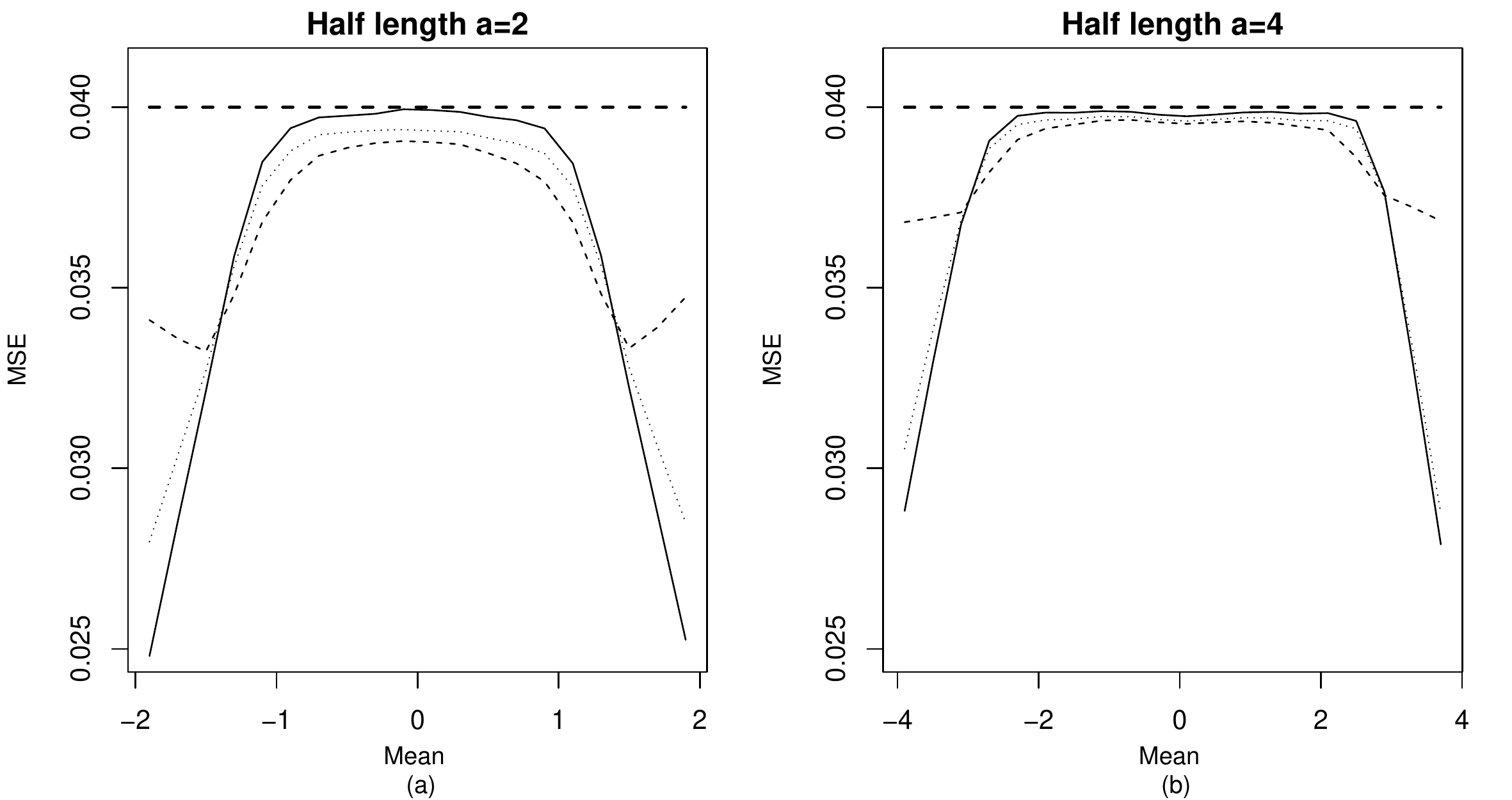}
       \caption{{MSE corresponding to different values of the restricted mean parameter $\mu$ with (a) $a=2$ and (b) $a=4$ and  the Bayes estimator $U_1$ (solid), $U_2$ (dashed) and $U_2^{\prime}$ (dotted) and simple mean estimator  $J$ (solid dashed). Results are based on $n=100$ and $10^4$ replications.} \label{normal100}}
     \end{figure}
     
The differences between all estimators are now smaller but the proposed estimators    $U_2$,  $U_2^{\prime}$ still provide benefit over $U_1$ for wide intervals in both cases. Again, the benefit is larger for a narrow interval of values. Interestingly, estimators $U_1$ and  $U_2^{\prime}$  now do not have a rising ``tails'' in the MSE on the bounds as the increased sample size allows to improve the estimation. At the same time, the MSE for estimator  $U_2$ still increases on the bounds. However, in contrast to the performance using $n=15$, the MSE associated with the estimator $U_2$ is now never above the MSE of the estimator $J$. 

Figure~\ref{normal1000} presents the MSE associated with the four considered estimators $J$, $U_1$, $U_2$,  $U_2^{\prime}$, for each value of $\mu$ in intervals (-a,a) using the total sample sizes $n=1000$, and $a=2$ and $a=4$.

\begin{figure}[!h]
  \centering
      \includegraphics[width=1\textwidth]{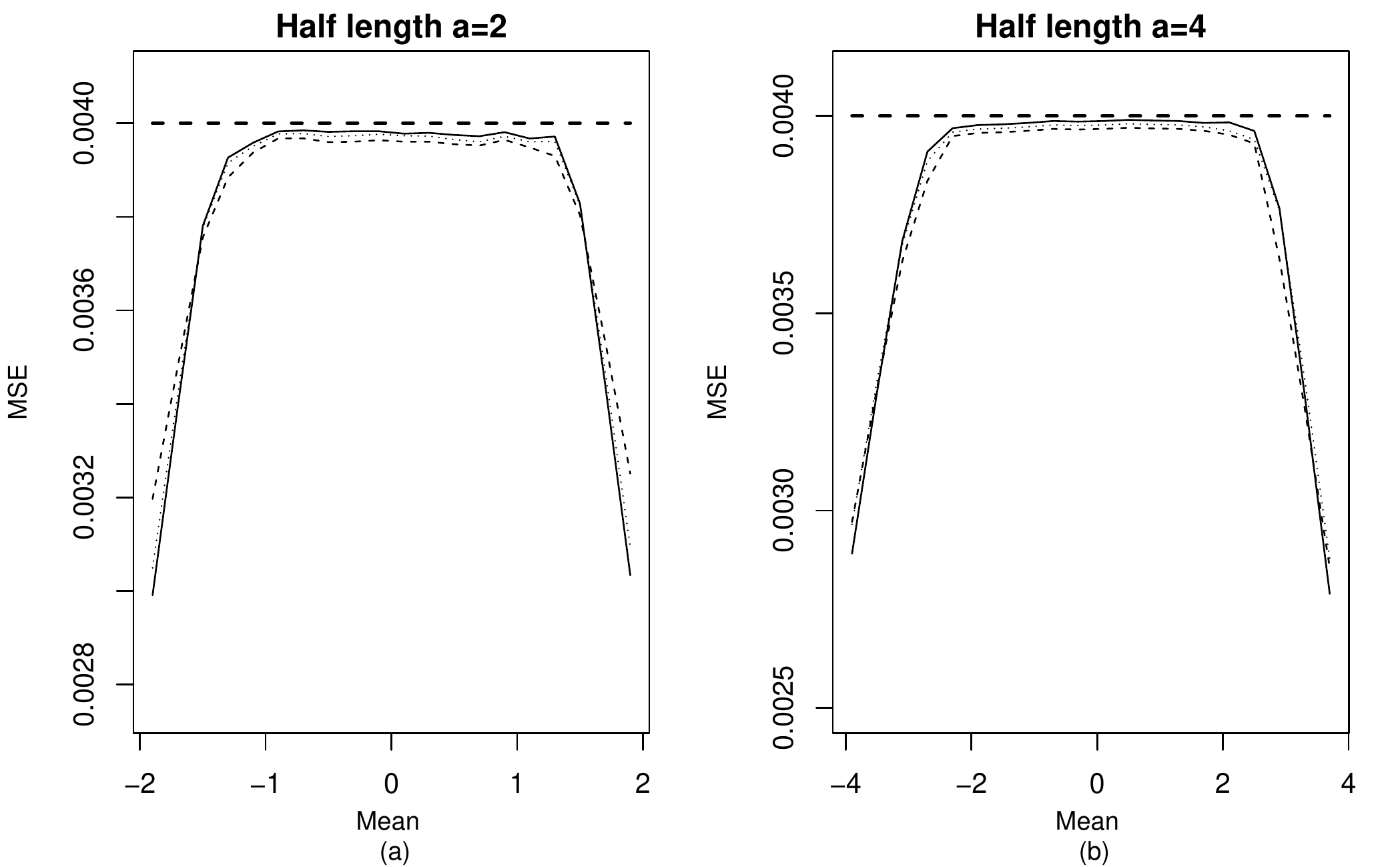}
       \caption{{MSE corresponding to different values of the restricted mean parameter $\mu$ with (a) $a=2$ and (b) $a=4$ and  the Bayes estimator $U_1$ (solid), $U_2$ (dashed) and $U_2^{\prime}$ (dotted) and simple mean estimator  $J$ (solid dashed). Results are based on $n=1000$ and $10^4$ replications.} \label{normal1000}}
     \end{figure}
     
For the large sample size, the same qualitative pattern stands while the differences in the MSE is smaller. Importantly, the MSE of the estimator $U_2$ does not increase on the tails and behaves similarly to the estimator $U_1$ and $U_2^\prime$.

     \subsection*{Bayesian Estimation of the Parameters of a Gamma Distribution}
The differences in the MSEs for both parameters for various values of  the true parameters $\alpha_1$ and $\alpha_2$ are given in Figure~\ref{gamma100} for the sample size $n=100$ and in Figure~\ref{gamma1000} for the sample size $n=1000$.

     \begin{figure}[p]
  \centering
       \includegraphics[width=1\textwidth]{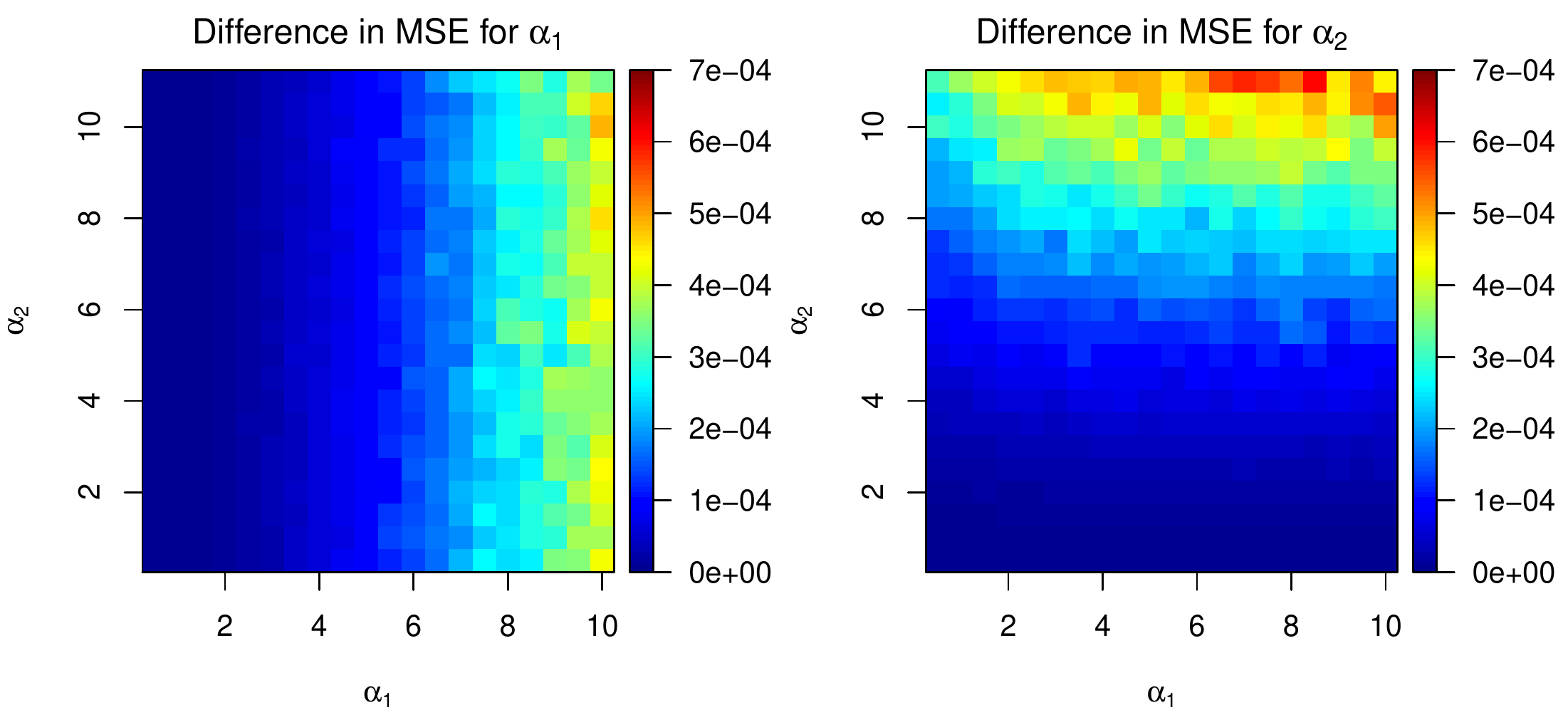}
       \caption{Difference in the MSEs for parameters $\alpha_1$ and $\alpha_2$ for their different true values and using Bayes estimator under the squared error loss function and Bayes estimator under $L_{sq}^{(2)}$. Results are based on $n=100$ and $10^4$ replications.  \label{gamma100}}
     \end{figure}

    \begin{figure}[p]
  \centering
       \includegraphics[width=1\textwidth]{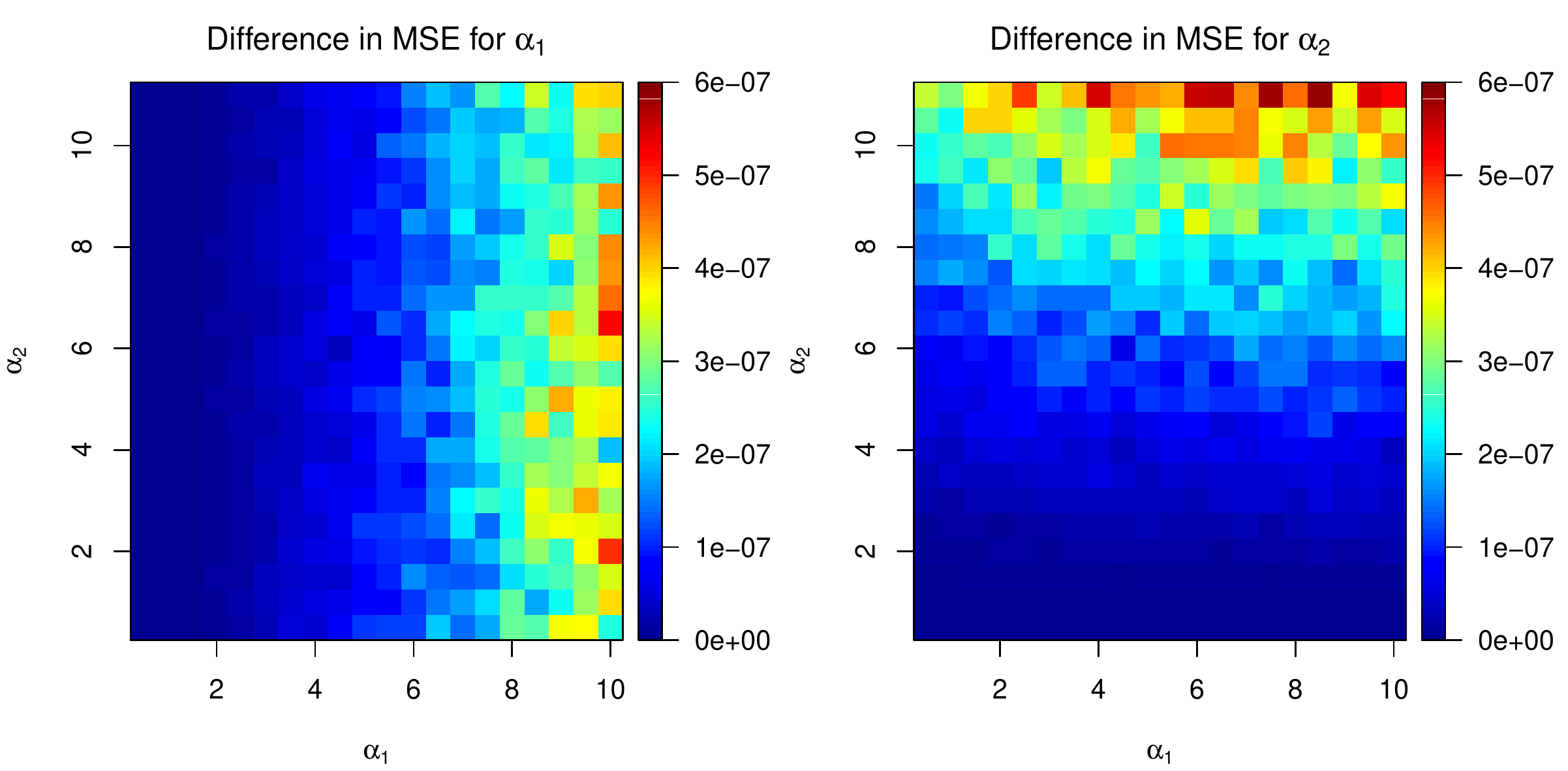}
       \caption{Difference in the MSEs for parameters $\alpha_1$ and $\alpha_2$ for their different true values and using Bayes estimator under the squared error loss function and Bayes estimator under $L_{sq}^{(2)}$. Results are based on $n=1000$ and $10^4$ replications.  \label{gamma1000}}
     \end{figure}
     
While the differences in the MSE is smaller as expected, the proposed estimator corresponds to lower values of the MSE for all considered values of the true parameter $\alpha_1$ and $\alpha_2$. This means that the proposed estimator can provide benefits in terms of the MSE for moderate and large sample sizes as well.

\subsection*{Bayesian Estimation of the Parameters of a Weibull Distribution}
The difference between the MSE of estimators for positive parameters $\lambda$, $\nu$  are given in Table~\ref{weibull100} for the sample size $n=100$, and in Table~\ref{weibull1000} for the sample size $n=1000$. In both cases, the MSE for $\nu$ are scaled by $\frac{1}{\nu^\lambda}$ and by $n$ to obtain the results on a similar scale for various parameters.

\begin{table}[ht!]
\centering
       \caption{Difference in the MSEs for parameters $\lambda$ (upper lines) and $\nu$ (lower lines) for their different true values and using Bayes estimator under the squared error loss function and Bayes estimator under $L_1^{(m)}$. The MSE for $\nu$ are scaled by $1/{\nu^\lambda}$  and $n=100$ to obtain the results on a similar scale for various parameters. Results are based on $10^4$ replications.  \label{weibull100}}
\begin{tabular}{cccccc}
  \hline
 &  $\nu=1$ & $\nu=2$ &  $\nu=5$ & $\nu=10$ & $\nu=15$ \\ 
  \hline
\multirow{2}{*}{$\lambda=1$} & 0.0105 &  0.0096 &  0.0062 &  0.0122 &  0.0103 \\ 
 &  0.0133 & 0.0048 & 0.0208 & 0.0132 & 0.0188 \\
    \hline

\multirow{2}{*}{$\lambda=2$} & 0.0533&  0.0417 & 0.0365  & 0.0430 & 0.0479 \\ 
&  0.0144 & 0.0049 &  0.0276 &  0.0285 &  0.0243 \\
   \hline

\multirow{2}{*}{$\lambda=3$} &  0.0943 & 0.1072 & 0.1138  & 0.1026  & 0.1095\\ 
 &  0.0143 &  0.0091  & 0.0237  & 0.0172  & 0.0101 \\
    \hline

\multirow{2}{*}{$\lambda=4$} &  0.1703  & 0.1830  & 0.1730 & 0.1852 & 0.1771\\ 
 &    0.0081  & 0.0166  & 0.0196 &  0.0065  & 0.0030 \\
    \hline

\multirow{2}{*}{$\lambda=5$} & 0.3456 & 0.2944 & 0.2603 &  0.2554 & 0.2445 \\ 
&  0.0094 & 0.0254 &  0.0110  & 0.0021 &  0.0008 \\
   \hline
\end{tabular}
\end{table}

\begin{table}[ht!]
\centering
       \caption{Difference in the MSEs for parameters $\lambda$ (upper lines) and $\nu$ (lower lines) for their different true values and using Bayes estimator under the squared error loss function and Bayes estimator under $L_1^{(m)}$. The MSE for $\nu$ are scaled by $1/{\nu^\lambda}$  and $n=1000$ to obtain the results on a similar scale for various parameters. Results are based on $10^4$ replications.  \label{weibull1000}}
\begin{tabular}{cccccc}
  \hline
 &  $\nu=1$ & $\nu=2$ &  $\nu=5$ & $\nu=10$ & $\nu=15$ \\ 
  \hline
\multirow{2}{*}{$\lambda=1$} &  0.0002  & 0.0002 & 0.0023 & 0.0005 &  0.0022 \\ 
 &  0.0022  & 0.0009 & 0.0022 &  0.0022&  0.0010\\
    \hline

\multirow{2}{*}{$\lambda=2$} & 0.0083 &  0.0061  & 0.0026  & 0.0069 & 0.0077 \\ 
& 0.0004  & 0.0003  & 0.0026  & 0.0017 &  0.0017 \\
   \hline

\multirow{2}{*}{$\lambda=3$} &  0.0126 & 0.0119 & 0.0062 & 0.0014 & 0.0178\\ 
 &  0.0006  & 0.0003 &  0.0026 &  0.0017  & 0.0006 \\
    \hline

\multirow{2}{*}{$\lambda=4$} &  0.0071  & 0.0105&  0.0150 & 0.0054&  0.0152\\ 
 &   0.0020 & 0.0018 & 0.0020 & 0.0006 & 0.0002 \\
    \hline

\multirow{2}{*}{$\lambda=5$} & 0.0162 & 0.0220 & 0.0117 & 0.0154 & 0.0013\\ 
&  0.0013 & 0.0034 & 0.0010 & 0.0001 & 0.0000 \\
   \hline
\end{tabular}
\end{table}

While the differences in the MSEs are smaller for both parameters,  they are still positive for all considered true values of $\lambda$ and $\nu$. It follows that the proposed Bayes estimator leads to a smaller MSE than the estimator corresponding to the squared error loss function for the moderate and large sample sizes.

\end{document}